\documentclass[a4paper,12pt]{article}

\usepackage[all,cmtip]{xy}
\usepackage{amssymb,amsmath,amsthm,enumerate,cite}
\usepackage{graphicx,caption,subcaption,color}
\renewcommand{\epsilon}{\varepsilon}
\newcommand{\R}{\mathbb{R}}
\newcommand{\Z}{\mathbb{Z}}
\newcommand{\C}{\mathbb{C}}

\renewcommand{\phi}{\varphi}
\newcommand{\xx}{\mathbf{x}}
\newcommand{\yy}{\mathbf{y}}

\newtheorem{thm}{Theorem}
\newtheorem{lemma}[thm]{Lemma}
\newtheorem{prop}[thm]{Proposition}

\theoremstyle{definition}
\newtheorem{definition}[thm]{Definition}

\theoremstyle{remark}
\newtheorem{remark}[thm]{Remark}

\title{Symplectic embeddings of the $\ell_p$-sum of two discs}

\author{Yaron Ostrover\thanks{School of Mathematical Sciences, Tel Aviv University, Ramat Aviv 6997801, Israel. } \ and Vinicius G. B. Ramos\thanks{Instituto de Matem\'atica Pura e Aplicada - Estrada Dona Castorina 110 - Rio de Janeiro - Brazil.}}

\date{}

\begin{document}

\maketitle
\begin{abstract}  
In this paper we study symplectic embedding questions for the  $\ell_p$-sum of two discs in ${\mathbb R}^4$, when $1 \leq p \leq \infty$. In particular, we compute the symplectic inner and outer radii in these cases, and show how different kinds of  embedding rigidity and flexibility phenomena  appear as a function of the parameter $p$.  
\end{abstract}

\section{Introduction and Results}

Since Gromov's seminal work~\cite{gromov}, questions about symplectic embeddings have lain at the core of symplectic geometry.
These questions are usually very difficult even for a relatively simple class of examples. In fact, only recently has it become possible to specify exactly when a four-dimensional ellipsoid is symplectically embeddable in a ball~\cite{McD-Sch}, or in another four-dimensional ellipsoid~\cite{McD}. For more information on 
symplectic embeddings  we refer the reader e.g., to the recent survey~\cite{Schlenk}. 

\medskip

In~\cite{bidisk}, using the theory of embedded contact homology, the second named author established sharp obstructions
 for symplectic embeddings of the  product of two Lagrangian discs in ${\mathbb R}^4$ into balls, ellipsoids and symplectic polydiscs. 
This product configuration appears naturally as the phase space of  billiard dynamics in a round disc (see e.g.,~\cite{Art-Ost,bidisk}). 
In this note we extend the above results to the family of $\ell_p$-sums of two Lagrangian discs. In particular, we show how different kinds of symplectic embedding rigidity and flexibility phenomena appear as functions of the parameter $1 \leq p \leq \infty$.
For comparison, we also consider similar embedding questions in the natural counterpart case of the $\ell_p$-sum of two symplectic discs. In order to be more precise and state our results, we first introduce some notations.
\medskip

Consider ${\mathbb R}^{4}$ equipped with coordinates $(x_1,x_2,y_1,y_2)$, and with the standard symplectic
form $\omega_0 = \sum_{i=1}^2 dx_i \wedge dy_i$. 
For two subsets $X_1,X_2\subset \R^4$, we write $X_1\hookrightarrow X_2$ if there is an embedding $\varphi:X_1\hookrightarrow X_2$ preserving the symplectic form, i.e., $\varphi^*\omega_0=\omega_0$. We denote the symplectic inner and outer radii of a set $X \subset {\mathbb R}^4$ by
\begin{equation*}
r_{S}(X)=\sup\left\{r\in\R\mid B^4(r)\hookrightarrow  X \right \} \ {\rm and} \ \ 
R_{S}(X)=\inf\left\{r\in\R\mid X \hookrightarrow B^4(r)\right\},
\end{equation*}
where $B^4(r) = \{z\in\R^4\mid\pi\Vert z\Vert^2 <r\}$ is the Euclidean ball 
with Gromov width $r$ (i.e., with radius $\sqrt{r/\pi}$).
Moreover, for $1 \leq p < \infty$, we denote by 
\begin{equation} \label{p-sum-of-discs}{\mathbb X}_p  
= \left\{(\xx,\yy)\in\R^2\times\R^2\mid\Vert\xx\Vert^p+\Vert\yy\Vert^p<1\right\},\end{equation}
the $\ell_p$-sum of two Lagrangian discs, where here $\| \cdot \|$ denotes the standard Euclidean norm on ${\mathbb R}^2$. If $p = \infty$ we set \[{\mathbb X}_{\infty} =\left\{(\xx,\yy)\in\R^2 \times \R^2 \mid\max(\Vert\xx\Vert,\Vert\yy\Vert)<1\right\}.\]
Our first result concerns the inner and outer radii of ${\mathbb X}_p$. 
More precisely, for $p \geq 1$, we  denote the area of the unit disk in $\R^2$ with respect to the standard $\ell_p$-norm by
\begin{equation} \label{eq:area-p-norm}
 A(p)=4\int_0^1(1-r^p)^{1/p}\,dr=\frac{4\cdot\Gamma(1+{\frac 1 p})^2}{\Gamma(1+{\frac 2 p})},
\end{equation} 
where $\Gamma (\cdot )$ is the Gamma function. Moreover, for $v\in[0,(1/4)^{1/p}]$ we define 
\begin{equation}\label{eq:g}
g_p(v):=2\int_{\left(\frac{1}{2}-\sqrt{\frac{1}{4}-v^p} \,\right)^{1/p}}^{\left(\frac{1}{2}+\sqrt{\frac{1}{4}-v^p}\, \right)^{1/p}}\sqrt{(1-r^p)^{2/p}-\frac{v^2}{r^2}}\,dr.
\end{equation}
\begin{thm}\label{thm:main} For $1 \leq p < \infty$, the symplectic inner radius of ${\mathbb X}_p$ is given by 
\begin{align*}
r_{S}({\mathbb X}_p)&=\left\{\begin{array}{ll}2\pi\left(1/4\right)^{1/p},&\text{ if }1\le p\le 2,\\A(p),&\text{ if  } p\ge 2,\end{array}\right.
\end{align*}
and the symplectic outer radius by
\begin{align*}
R_{S}({\mathbb X}_p)&=\left\{\begin{array}{ll}A(p),&\text{ if }1\le p\le 2,\\
2\pi\left(1/4\right)^{1/p},&\text{ if }2\le p \le 9/2,\\
2\pi \left(g_p'\right)^{-1}(-2\pi/3)+3g_p\left(\left(g_p'\right)^{-1}(-2\pi/3)\right),&\text{ if } 9/2< p<\infty. \end{array}\right.
\end{align*}
\end{thm}
\medskip
\begin{remark} {\rm
It is shown in Proposition~\ref{lem:properties-of-gp} below that $g_p'$ is injective and its image contains $-2\pi/3$ for $p\ge 9/2$. Thus the expression in the last line of Theorem \ref{thm:main} is well defined. We also recall that in~\cite{bidisk} it was proved that $r_{ S}({\mathbb X}_\infty) = 4$ and $R_{ S}({\mathbb X}_\infty)=3\sqrt{3}$. It is clear that $A(p)\to 4$ as $p\to \infty$, and we will see below that $R_{\cal S}({\mathbb X}_p) \to 3\sqrt{3}$ as $p\to\infty$.
Moreover, it can be shown explicitly that the functions defined by the formulas above are continuous at $p=2$ and $p=9/2$. This should indeed be the case, since $r_{S}({\mathbb X}_p)$ and $R_{S}({\mathbb X}_p)$ are clearly continuous. Finally, we remark that at  $p=9/2$, a certain ``phase-transition'' occurs between rigidity and flexibility of the embedding ${\mathbb X_p} \hookrightarrow B^4(r)$, 
as explained in Theorem~\ref{thm:rigid} below. 
}
\end{remark}

\medskip

The proof of Theorem \ref{thm:main} follows the approach of \cite{bidisk}, i.e., we use the theory of integrable Hamiltonian systems to describe the domain ${\mathbb X}_p$ as a toric domain, and then use the machinery of ECH capacities  to find symplectic embedding obstructions. 
To find optimal symplectic embeddings on the other hand, we combine results from~\cite{busepin},~\cite{cgcc},~\cite{lili}, and~\cite{mcduffblowup}.  
We turn now to explain this in more details.

\subsection{The $\ell_p$-sum of two Lagrangian discs as a toric domain}
Toric domains form a large class of symplectic manifolds that generalizes ellipsoids and polydiscs. For this class, certain symplectic embedding questions are better understood, particularly in dimension four, see, e.g.,~\cite{mcdel,cgcc,ccfhr}. For our purposes,  a domain $X\subset \R^4$ is toric if it is invariant under the standard action of ${\mathbb T}^2$ on $\R^4 \simeq \C^2$. Here, we allow $X$ to have boundary and corners. Note that a toric domain $X$ is completely determined by its image under the moment map $\mu:\C^2 \simeq \R^4\to \R^2$, given by $\mu(z_1,z_2)=(\pi|z_1|^2,\pi|z_2|^2)$, where $z_j = x_j + iy_j$ for $j=1,2$. For a domain $\Omega\subset\R_{\ge 0}^2$, we denote the corresponding toric domain $\mu^{-1}(\Omega)$ in ${\mathbb R}^4$ by $ X_{\Omega}$. 

\medskip

From now on we assume that $\Omega \subset\R_{\ge 0}^2$ is the region bounded by the coordinate axes, and the graph of a decreasing continuous function $\phi:[0,a]\to\R_{\ge 0}$ such that $\phi(a)=0$.

\begin{definition}
With the above notations, the set $X_\Omega$ is called a {\em concave} or {\em convex} toric domain if the function $\phi$ is convex or concave, respectively.
\end{definition}
\begin{remark}
In the literature there are more general definitions of convex/concave toric domains including, e.g., rectangles touching the origin. However, we will not need to use them in the context of this paper, c.f. \cite{cgcc,hutbey}.
\end{remark}

The main ingredient  in the proof of Theorem~\ref{thm:main} above is the following result.
\begin{thm}\label{thm:toric}
For $p \in[1,\infty)$, the interior of the Lagrangian product ${\mathbb X}_p$  is symplectomorphic to the interior of the toric domain $X_{\Omega_p}$, where $\Omega_p$ is the relatively open set in $\R_{\ge0}^2$ bounded by the coordinate axes and the curve parametrized by
\begin{equation}\label{eq:curve}\begin{aligned}\left(2\pi v+g_p(v),g_p(v)\right),&\quad {\it for} \,\, v \in[0,(1/4)^{1/p}],\\
\left(g_p(-v),-2\pi v+g_p(-v)\right),&\quad {\it for} \,\, v\in[-(1/4)^{1/p},0],
\end{aligned}\end{equation}
where $g_p : [0,(1/4)^{1/p}] \rightarrow \R$ is the function defined by \eqref{eq:g} above.
\end{thm}
The following analogous result for ${\mathbb X}_\infty$ was shown in \cite{bidisk}.
\begin{thm}[{{\cite[Theorem 3]{bidisk}}}]
The domain ${\mathbb X}_\infty$ is symplectomorphic to the toric domain $X_{\Omega_\infty}$, where $\Omega_\infty$ is the relatively open set in $\R_{\ge0}^2$ bounded by the coordinate axes and the curve parametrized by
\begin{equation}\label{eq:curve2}
2\left(\sqrt{1-v^2}+v(\pi-\arccos v),\sqrt{1-v^2}-v\arccos v\right),\quad {\it for} \, \, v\in[-1,1].
\end{equation}
\end{thm}

\begin{remark}\label{rmk:refl}
Note that the curve \eqref{eq:curve} is invariant under the reflection about the line $y=x$.
Moreover, a simple calculation shows that \eqref{eq:curve} converges to \eqref{eq:curve2} as $p \to\infty$.
\end{remark}

With some additional computational work, we further claim that:
\begin{prop}\label{prop:cc}
The toric domain $X_{\Omega_p}$ defined in Theorem~\ref{thm:toric} above is convex for $p\in[1,2]$, and concave for $p\in[2,\infty]$.
\end{prop}

Figure \ref{fig:cc} shows the set $\Omega_p$ for $p=1,2,6$. 
One can directly check that $\Omega_2$ is a right triangle, which reflects the fact that ${\mathbb X}_2$ is the Euclidean ball. 

\begin{figure}
\centering
\begin{subfigure}[t]{0.32\textwidth}
\centering
\includegraphics[scale=0.7]{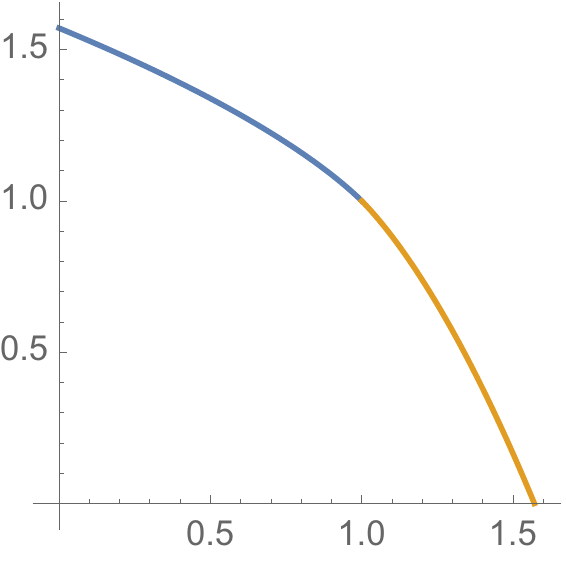}
\caption{$p=1$}
\end{subfigure}
\begin{subfigure}[t]{0.32\textwidth}
\centering
\includegraphics[scale=0.7]{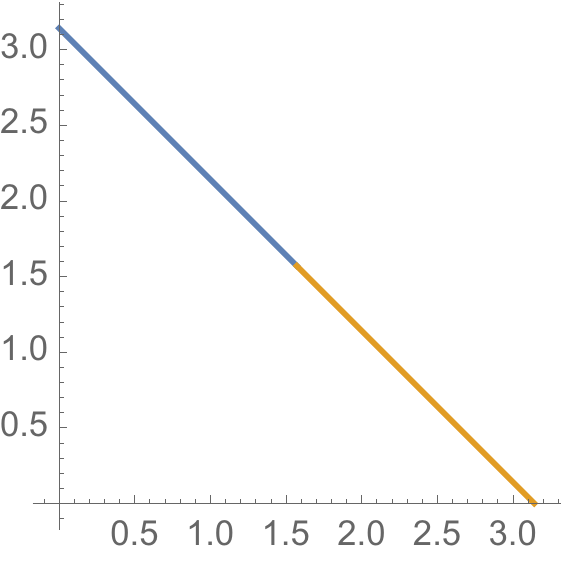}
\caption{$p=2$}
\end{subfigure}
\begin{subfigure}[t]{0.32\textwidth}
\centering
\includegraphics[scale=0.7]{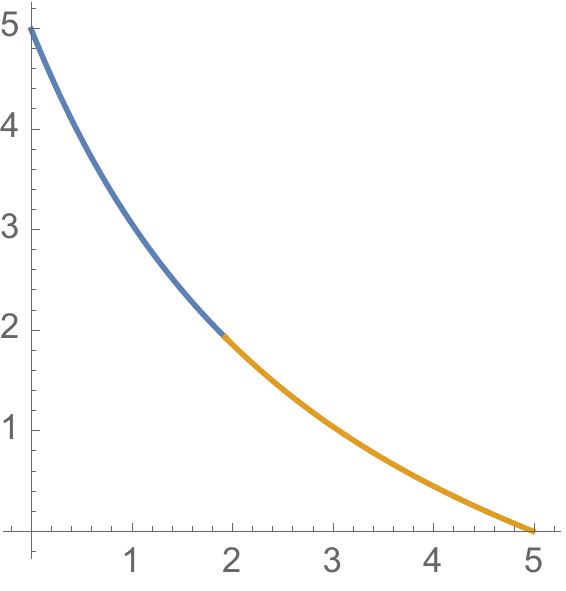}
\caption{$p=6$}
\end{subfigure}
\caption{The set $\Omega_p$ for different values of $p$}\label{fig:cc}
\end{figure}

\subsection{The rigidity and flexibility of the embeddings}

In this section we discuss certain rigidity and flexibility phenomena of the embeddings described in Theorem~\ref{thm:main}, and explain the significance of the specific values of $p$ appearing in that theorem.  
We start with the following notions of symplectic embedding rigidity, which for the purpose of this paper we state only in $\R^4$. 
\begin{definition}
Let $X_1$ and $X_2$ be subdomains of $\R^{4}$. 
\begin{enumerate}[(a)]
\item The symplectic embedding problem $X_1\overset{?}{\hookrightarrow}X_2$
is said to be {\it rigid} if \[X_1\hookrightarrow  r  X_2\iff X_1\subseteq r  X_2.\]
\item The symplectic embedding problem $X_1\overset{?}{\hookrightarrow}X_2$
is said to be {\it torically rigid} if the embedding $\widetilde{X}_1\overset{?}{\hookrightarrow}\widetilde{X}_2$ is rigid, where $\widetilde{X}_1$ and $\widetilde{X}_2$ are two toric domains whose interiors are symplectomorphic to the interiors of $X_1$ and $X_2$, respectively.

\item The symplectic embedding problem $X_1\overset{?}{\hookrightarrow}X_2$
is said to be {\it non-rigid} if it is neither rigid, nor torically rigid.
\end{enumerate}
\end{definition}
With this terminology, we can now state the rigidity of the symplectic embeddings  from Theorem \ref{thm:main} as follows.
\begin{thm}\label{thm:rigid} Let $B \subset \R^4$ be a Euclidean ball, and let ${\mathbb X}_p$ be the $\ell_p$-sum of two Lagrangian discs given by~$(\ref{p-sum-of-discs})$. Then, 
\quad
\begin{enumerate}[(a)]
\item The symplectic embedding $B \overset{?}{\hookrightarrow} {\mathbb X}_p$ is rigid for $1 \leq p \leq 2$.

\item The symplectic embedding $B \overset{?}{\hookrightarrow} {\mathbb X}_p$ is torically rigid for all $p\ge 1$.

\item The symplectic embedding ${\mathbb X}_p \overset{?}{\hookrightarrow} B$ is rigid for $2 \leq p \leq 9/2$.
\item 
The symplectic embedding ${\mathbb X}_p \overset{?}{\hookrightarrow} B$ is torically rigid for $1 \leq p \leq 9/2$.

\item 
The symplectic embedding ${\mathbb X}_p \overset{?}{\hookrightarrow} B$ is non-rigid for $p>9/2$.
\end{enumerate}
\end{thm} 
\begin{remark} {\rm
From our point of view, the most surprising part of Theorem \ref{thm:rigid} is the change of behavior of the embedding question ${\mathbb X}_p \overset{?}{\hookrightarrow} B$ at the value $p=9/2$. Note that there is no change of convexity of the toric image at this value of $p$, and that there is no a priori reason for this transition from rigidity to flexibility of the embedding. We refer the reader to Lemma~\ref{lemma:c1c2} and Proposition~\ref{prop:cap} below for more details on the appearance of the value $p=9/2$ in this setting.

}
\end{remark}

\medskip

\subsection{The $\ell^p$-sum of two symplectic discs}
As a natural counterpart of the above results, in this section
we discuss the analogues of Theorems \ref{thm:main} and \ref{thm:rigid} for the $\ell^p$-sum of two equal-size symplectic discs in ${\mathbb R}^4$. More precisely, for $p > 0$ let
\[{\mathbb B}_p({\mathbb C}^2)=\left\{(z_1,z_2)\in\C^2\mid \pi^{p/2}\left(\Vert z_1 \Vert^p+\Vert z_2\Vert^p\right)<1\right\}.\]
Note that 
as $p\to\infty$, the set ${\mathbb B}_p({\mathbb C}^2)$ converges to the symplectic polydisc $$ {\mathbb B}_{\infty}({\mathbb C}^2) := \{ (z_1,z_2)\in\C^2\mid \max \{ \pi\|z_1\|^2, \pi \|z_2\|^2 \} <1\}.$$ 
We remark that in the symplectic literature  ${\mathbb B}_{\infty}({\mathbb C}^2) $ is usually denoted by $P(1,1)$.
It  follows directly from the definition that ${\mathbb B}_p({\mathbb C}^2)$  is a toric domain $X_{\Lambda_p}$, where $\Lambda_p$ is the relatively open set in $\R_{\ge 0}^2$ bounded by the coordinate axes and the curve $$x^{p/2}+y^{p/2}=1 \ \ ({\rm or \ } \max \{ \|x\|, \|y\| \} =1 \ {\rm for \ } p=\infty).$$ Thus, ${\mathbb B}_p({\mathbb C}^2)$ is concave if $0<p\le 2$, and convex if $p\ge 2$. Moreover, it is clear that in this case the notions of rigidity and toric rigidity coincide. 

\begin{thm}\label{thm:symplp}
Let $B \subset \R^4$ be the Euclidean unit ball. Then
\begin{enumerate}[(a)]
\item The symplectic embedding  $B \overset{?}{\hookrightarrow} {\mathbb B}_p({\mathbb C}^2)$ is rigid for all $p \geq 1$, and \[B(c)\hookrightarrow {\mathbb B}_p({\mathbb C}^2) \iff  c \leq \min \{ 1,2^{1-2/p} \}.\]

\item The symplectic embedding  ${\mathbb B}_p({\mathbb C}^2) \overset{?}{\hookrightarrow} B$ is rigid for $p\ge 2$, and  
\[{\mathbb B}_p({\mathbb C}^2)\hookrightarrow B(c)\iff c\ge 2^{1-2/p}.\]

\item 
The symplectic embedding $ {\mathbb B}_p({\mathbb C}^2)\overset{?}{\hookrightarrow} B$ is non-rigid for $1\le p< 2$ and
\[{\mathbb B}_p({\mathbb C}^2)\hookrightarrow B(c)\iff c\ge 
\Bigl (1 + 2^{\frac{p}{p - 2}} \Bigr)^{1-2/p}.\]
\end{enumerate}

\end{thm}

In the last case, we can actually prove a little more. Consider the ellipsoid
\[E(a,b)=\left\{(z_1,z_2)\in\C^2\mid \pi\left(\frac{|z_1|^2}{a}+\frac{|z_2|^2}{b}\right)<1\right\}.\]

\begin{prop}\label{prop:lpellip} With the above notations, 
\[{\mathbb B}_1({\mathbb C}^2) \hookrightarrow E(a,b)\iff \min(a,b)\ge 1/2 \ {\rm and} \ \max(a,b)\ge 2/3 .\]
\end{prop}

\begin{remark}
It is not hard to check that ${\mathbb B}_1({\mathbb C}^2)$ and $E(1/2,2/3)$ have the same volume, and thus ${\mathbb B}_1({\mathbb C}^2)\hookrightarrow E(1/2,2/3)$ is a volume filling embedding.
Proposition \ref{prop:lpellip} can be extended to prove that there exists an embedding $${\mathbb B}_p({\mathbb C}^2)\hookrightarrow E \left (2^{1-2/p},\left (1 + 2^{\frac{p}{p - 2}} \right)^{1-2/p}\right)$$ for all $p\in[1,2)$, but for $p\in(1,2)$ this embedding is not volume filling. Since the proof of the existence of this embedding for $p\in(1,2)$ is very similar to the case $p=1$, although much more tedious, we omit it here. 
\end{remark}

\vskip 4pt

\noindent{\bf Structure of the paper:} In Section~\ref{Sec-Integrable-systems} we use the integrability of the Hamiltonian system associated with ${\mathbb X}_p$ to prove Theorem \ref{thm:toric}. 
In Section~\ref{sec:cap} we recall some relevant definitions and results concerning the ECH capacities, and in particular compute the first two capacities of ${\mathbb X}_p$. Finally, in Section~\ref{sec:proof-main-results} we prove Theorems~\ref{thm:main},~\ref{thm:rigid}, and~\ref{thm:symplp}, as well as Proposition~\ref{prop:lpellip}.

\vskip 8pt

\noindent{{\bf Acknowledgments:}} 
Part of this work was done while the first named author was at the Institute for Advanced Study in Princeton. YO is grateful to the IAS for the kind hospitality and support. YO is partially supported by the European Research Council starting grant No.~637386, by the ISF grant No.~667/18, and by the IAS School of Mathematics. The second named author is partially supported by a grant from the Serrapilheira Institute, the FAPERJ grant Jovem Cientista do Nosso Estado and the CNPq grant~305416/2017-0.

\section{Integrable systems and toric domains} \label{Sec-Integrable-systems}
In this section we prove Theorem \ref{thm:toric}, which is the main ingredient in the proof of Theorem~\ref{thm:main}. We start by recalling the classical  Arnold-Liouville theorem and the construction of action-angle coordinates \cite{arnold}. In what follows, if $B$ is an open set in $\R^n$, by abuse of notations we will denote by $\omega_0$ also the standard symplectic form on $B\times \mathbb{T}^n$, i.e., $\omega_0=\sum_i d\rho_i\wedge d\theta_i$, where $(\rho_1,\dots,\rho_n)$ and $(\theta_1,\dots,\theta_n)$ are the coordinates on $\R^n$ and on the torus $\mathbb{T}^n\cong\R^n/\Z^n$, respectively.

\begin{thm}[Arnold-Liouville]\label{thm:al}
Let $(M^{2n},\omega)$ be a symplectic manifold, and let $F=(H^1,\dots,H^n):M\to \R^n$ be a $C^\infty$-function whose components Poisson commute, i.e., $\{H^i,H^j\}=0$ for all $1 \leq i,j \leq n$.
\begin{enumerate}[(a)]
\item If $c\in\R^n$ is a regular value of $F$, i.e., the differentials of $H_1,\ldots,H_n$ are independent on $F^{-1}(c)$, and $F^{-1}(c)$ is compact and connected, then $F^{-1}(c)\cong \mathbb{T}^n$.
\item Let $U\subset M$ be an open set such that $F(U)$ is simply-connected, and does not contain critical values. For $c\in F(U)$, let $\{\gamma_1^c,\dots,\gamma_n^c\}$ be a set of simple closed curves generating $H_1(F^{-1}(c);\Z)$ that depend smoothly on $c$, and suppose that $\omega$ has a primitive $\lambda$ on $U$. Consider the function $\phi : {\mathbb R}^n \to {\mathbb R}^n$ defined by \begin{equation}\label{eq:phi}\phi(c)=\left(\int_{\gamma_1^c}\lambda,\dots,\int_{\gamma_n^c}\lambda\right).\end{equation}
Then $\phi$ is a diffeomorphism with its image $B$, and there exists a symplectomorphism $\Phi:(U,\omega)\to (B\times\mathbb{T}^n,\omega_0)$ such that the following diagram commutes.
\[
\xymatrix{
 \quad U\quad\ar[d]^F \ar[r]^{\Phi} & \;B \times\mathbb{T}^n \ar[d]^{\text{$\pi_1$}}\; \\
\;F(U)\; \ar[r]^{\phi}          & \quad B \quad }
\]
where here $\pi_1$ denotes the projection onto the first factor. 
\end{enumerate}
\end{thm}
\begin{remark}
The original result in \cite{arnold} states that if $c\in\R^n$ is a regular value of $F$ that satisfies the assumptions in (a), then it always has a neighborhood $U$ satisfying the assumptions in (b). So the conclusion of Theorem \ref{thm:al}(b) holds in a neighborhood of $c$. On the other hand, we remark that if one can extend the family of curves $\gamma_1^c, \ldots, \gamma_n^c$, while still maintaining the assumptions in (b), the diffeomorphisms $\phi$ and the symplectomorphism $\Phi$ can be extended too.
\end{remark}

\subsection{The toric picture of the Lagrangian $\ell_p$-sum}

Fix $p \in [1,\infty)$. 
For $(\xx,\yy)\in\R^2 \oplus \R^2$, a natural defining Hamiltonian function for the Lagrangian $p$-sum ${\mathbb X}_p$ \eqref{p-sum-of-discs} is the function $H_p : {\mathbb R}^4 \to \R$ given by  
$$H_p(\xx,\yy):=\Vert \xx\Vert^p+\Vert \yy\Vert^p. $$
Note that $\partial {\mathbb X}_p = H_p^{-1}(1)$, and while for $p\geq 2$ the function $H_p$ is $C^1$, it is not differentiable for $1 < p < 2$. 
Thus, in order to use the Arnold-Liouville theorem stated above, we first approximate $H_p$ by a sequence of smooth functions. 
We write  $H_p$ as $H_p = H_p^2 \circ H_p^1$, where $H_p^1 : \R^4 \to \R^2$ and $H_p^2 : \R^2 \to \R$ are given by
\begin{equation*}
H_p^1 (\xx,\yy) = (\|\xx\|^2,\|\yy\|^2), \quad {\rm and} \quad H_p^2 (s,t) = s^{p/2} + t^{p/2}.
\end{equation*}
Note that  $H_p^1 \in C^{\infty}(\R^4)$, and that $H_p^2$ is smooth away from the coordinate axes.
We approximate $H_p$ by a family of smooth Hamiltonian functions 
\[ H_p^{\epsilon}(\xx,\yy) : = H_p^{2,\epsilon} \circ H_p^1(\xx,\yy)  ,\] where the function $H_p^{2,\epsilon} : {\mathbb R}^2_{\ge 0}\setminus\{0\} \to \R$ is defined as follows. For $\epsilon>0$ small, let
\begin{align*}
f^\epsilon(t)&:=\left\{\begin{array}{ll} \alpha^\epsilon(t),&\text{ for }  0\le t\le \epsilon,\\
(1-t^{p/2})^{2/p},&\text{ for } \epsilon \le t \le 1-\epsilon, \\
\beta^\epsilon(t),&\text{ for } 1-\epsilon< t<1, \end{array}\right.
\end{align*}
where $\alpha^\epsilon$ and $\beta^\epsilon$ are smooth decreasing functions with $\alpha^\epsilon (0)=1$, $\beta^\epsilon(1)=0$, $(\alpha^\epsilon)'(0)<0$, and such that \[     \epsilon<\widetilde{\epsilon}\Rightarrow\alpha^\epsilon\ge \alpha^{\widetilde{\epsilon}}\ \text{and}\, \ \beta^\epsilon\ge \beta^{\widetilde{\epsilon}},\] and such that the function $t\mapsto tf^\epsilon(t)$ is strictly monotone for all $t\in[0,\epsilon]\cup[\epsilon,1-\epsilon]$. In particular, $t\mapsto tf^\epsilon(t)$ has a unique critical point at $t=1/2^{2/p}$, where the function $tf^\epsilon(t)$ attains its global maximum. For $(s,t)\in\R^2_{\ge 0}\setminus\{0\}$, we now define $$H_p^{2,\epsilon}(s,t)=\lambda^{-p/2},$$ where $\lambda$ is the unique number in $\R_{>0}$ such that $\lambda s= f^\epsilon(\lambda t)$. Note that one has  $H_p^{2,\epsilon} \in C^{\infty}(\R^2_{\ge 0} \setminus \{0\})$. Finally, we set $H_p^\epsilon=H_p^{2,\epsilon}\circ H_p^1\in C^{\infty}(\R^4 \setminus \{0\})$.

\medskip

Next, we  denote by $V : \R^4 \rightarrow \R$ the standard ``angular momentum'' given by
$$V(\xx,\yy) :=y_1x_2-y_2x_1.$$
The following propositions shows, roughly speaking, that the dynamical system associated with ${\mathbb X}_p$ is ``integrable'' in the sense of Theorem~\ref{thm:al} above. More precisely,

\begin{prop} \label{prop:properties-of-Hamiltonian-l_p}
Let $F^\epsilon=(H_p^\epsilon,J):\R^4\setminus\{0\}\to\R^2$. Then 
\begin{enumerate}[(a)]
\item $\{H_p^\epsilon,V\}=0$.
\item The image of $F^\epsilon$ consists of all points $(h,v)\in\R_{\ge 0}\times\R\setminus\{(0,0)\}$ such that 
$|v|\le\left(\frac{h}{2}\right)^{2/p},$ with equality occurring  if and only if $(h,v)$ is a critical value.
\item If $|v|<\left(\frac{h}{2}\right)^{2/p}$, then $(F^\epsilon)^{-1}(h,v)$ is compact and connected.
\end{enumerate}
\end{prop}
\begin{proof}
\begin{enumerate}[(a)]
\item Let $(r,\theta)$ be the polar coordinates for $\mathbf{y}$, and let $(p_r,p_{\theta})$ be the associated coordinates for $\mathbf{x}$. In particular, \[\Vert \mathbf{y}\Vert^2=r^2, \qquad\Vert\mathbf{x}\Vert^2= p_r^2+\frac{p_\theta^2}{r^2},\qquad \yy\times\xx=p_\theta,\]
and  \[\begin{aligned}H_p^\epsilon(\xx,\yy)&=H_p^{2,\epsilon}\left(\Vert x\Vert^2,\Vert y\Vert^2\right)=H_p^{2,\epsilon}\left(p_r^2+\frac{p_\theta^2}{r^2},r^2\right), \quad \text{and} \quad
V(\xx,\yy)&=p_\theta.\end{aligned}\]
Consequently one has $\{H_p^\epsilon,V\}=0$.

\item Suppose that $H_p^{\epsilon}(\xx,\yy)=h$, and $V(\xx,\yy)=v$. It follows from the definition of the function $H_p^\epsilon$ that
\begin{equation}\label{eq:hpe}
h^{-2/p}\Vert \xx\Vert^2=f^\epsilon(h^{-2/p}\Vert \yy\Vert^2).
\end{equation}

By the assumptions on $f^\epsilon$, the maximum of the function $t\mapsto tf^\epsilon(t)$ is $1/2^{4/p}$, which is attained at $t=1/2^{2/p}$. So it follows from \eqref{eq:hpe} that
\[\begin{aligned}v^2&\le \Vert \xx\Vert^2\Vert \yy\Vert^2=h^{2/p}\Vert\yy\Vert^2 f^\epsilon(h^{-2/p}\Vert\yy\Vert^2)
&\le \left(\tfrac{h}{2}\right)^{4/p}.\end{aligned}\]
Moreover, the extremal values of $v$ are attained if and only if $\xx$ is orthogonal to $\yy$, and 
$h^{-2/p}\Vert \yy\Vert^2$ is the point of maximum of $t \mapsto t f^{\epsilon}(t)$.
We next compute the gradients $\nabla H_p^\epsilon$ and $\nabla V$. Since $H_p^{2,\epsilon}$ is homogeneous, it follows that $$\nabla H_p^\epsilon(\xx,\yy)=c \left( \xx,-(f^{\epsilon})'(h^{-2/p}\Vert \yy\Vert^2)\yy\right),$$ for some $c>0$. On the other hand, an easy calculation gives $$\nabla V(\xx,\yy)=(J_0\yy,-J_0\xx), \ \text{where} \ J_0=\begin{bmatrix}0&-1\\1&0\end{bmatrix}.$$ It is clear that $\nabla H_p^\epsilon$ and $\nabla V$ never vanish. Thus, $(\xx,\yy)$ is a critical point of $F^\epsilon$ if and only if $\nabla H_p^\epsilon(\xx,\yy)$ and $\nabla V(\xx,\yy)$ are parallel. This happens if and only if
\begin{equation}\label{eq:hpe2}
\xx=\pm\sqrt{-(f^\epsilon)'(h^{-2/p}\Vert \yy\Vert^2)} J_0 \yy.
\end{equation}
Using \eqref{eq:hpe2} we conclude that \eqref{eq:hpe} is equivalent to requiring that $\xx$ is orthogonal to $\yy$, and  
\begin{equation*}
h^{2/p} f^\epsilon(h^{-2/p}\Vert\yy\Vert^2)=-(f^\epsilon)'(h^{-2/p}\Vert \yy\Vert^2)\Vert\yy\Vert^2,
\end{equation*}
i.e., $h^{-2/p} \Vert\yy\Vert^2$ is a critical point of $t f^{\epsilon}(t)$. By assumption this function has a single critical point. So \eqref{eq:hpe2} holds if and only if $\xx$ is orthogonal to $\yy$ and $h^{-2/p}\Vert \yy\Vert^2$ is the point of maximum of $t f^{\epsilon}(t)$. Therefore $|v|=\left(\frac{h}{2}\right)^{2/p}$ if and only if $(h,v)$ is a critical value as required.

\item First observe that $(F^\epsilon)^{-1}(h,v)$ is a closed set in $\R^2$ which is contained in a bounded set, and so is compact. Next, let $(\xx,\yy)\in(F^\epsilon)^{-1}(h,v)$. First suppose that $v=0$. So $\xx$ and $\yy$ are parallel, and at least one of the two is nonzero. Through scaling, one can connect $(\xx,\yy)$ to $(h^{1/p}\mathbf{z},0)$, where $\mathbf{z}=\xx/\Vert\xx\Vert$ or $\mathbf{z}=\yy/ \Vert\yy\Vert$. Similarly, if $(\widetilde{\xx},\widetilde{\yy})\in(F^\epsilon)^{-1}(h,0)$, it can also be connected to a point of the form $(h^{1/p}\widetilde{\mathbf{z}},0)\in\R^2\times\R^2$ where $\Vert\widetilde{\mathbf{z}}\Vert=1$. Now through a rotation in $\R^2\times\{0\}$ we can connect $(h^{1/p}\mathbf{z},0)$ to $(h^{1/p}\widetilde{\mathbf{z}},0)$ while staying in $(F^\epsilon)^{-1}(h,0)$. Next suppose that $v\neq 0$, and let $(\xx,\yy)\in(F^\epsilon)^{-1}(h,v)$. Using polar coordinates $(p_r,p_\theta,r,\theta)$, it follows from 
\eqref{eq:hpe} that
\begin{equation}\label{eq:pr}p_r^2=h^{2/p}f^{\epsilon}(h^{-2/p}r^2)-\frac{v^2}{r^2}.\end{equation}
Moreover, any quadruple $(p_r,p_\theta,r,\theta)$ represents a point in $(F^{\epsilon})^{-1}(h,v)$ if $v=p_\theta$ and $(r,p_r)$ satisfy \eqref{eq:pr}. Now if $(p_r,p_\theta,r,\theta)$ and $(\widetilde{p}_r,\widetilde{p}_\theta,\widetilde{r},\widetilde{\theta})$ are the polar coordinates of two points in $(F^{\epsilon})^{-1}(h,v)$, we construct a path between them as follows. First, note that $v=p_\theta=\widetilde{p}_\theta$. Let $\{\theta(\tau)\}_{\tau\in[0,1]}\subset\R/2\pi\Z$ be a path connecting $\theta$ to $\widetilde{\theta}$. 
We now define a path $\{r(t)\}_{t\in[0,1]}\in\R_{>0}$ as follows. 
Let $\text{sign}:\R\to \{-1,0,1\}$ denote the sign function, and let $r_{max}$ be the largest solution of the equation 
\begin{equation}\label{eq-v-sqaure}
h^{2/p}r^2 f^{\epsilon}(h^{-2/p}r^2)=v^2.
\end{equation} For $t\in[0,1/2]$, let $r(t)$ be the affine path connecting $r$ to $r_{max}$, and for $t\in[1/2,1]$, let $r(t)$ be the affine path connecting $r_{max}$ to $\widetilde{r}$. By definition $r(t)$ is continuous. Next define
\[p_r(t)=\left\{\begin{aligned}\text{sign}(p_r)\sqrt{h^{2/p}f^{\epsilon}(h^{-2/p}r(t)^2)-\frac{v^2}{r(t)^2}},\quad&\text{if }t\in[0,1/2],\\
\text{sign}(\widetilde{p}_r)\sqrt{h^{2/p}f^{\epsilon}(h^{-2/p}r(t)^2)-\frac{v^2}{r(t)^2}},\quad&\text{if }t\in[1/2,1].\end{aligned}\right.\]
The function $p_r(t)$ is well defined and continuous at $t=1/2$ because $p_r(1/2)=0$ with either definition. Thus, $(p_r(t),\theta,r(t),\theta(t))$ represents a path connecting $(p_r,p_\theta,r,\theta)$ to $(\widetilde{p}_r,\widetilde{p}_\theta,\widetilde{r},\widetilde{\theta})$, which shows that $(F^\epsilon)^{-1}(h,v)$ is connected. 
\end{enumerate}
This completes the proof of Proposition~\ref{prop:properties-of-Hamiltonian-l_p}.
\end{proof}
Equipped with Theorem \ref{thm:al} and Proposition~\ref{prop:properties-of-Hamiltonian-l_p}, we can now prove Theorem \ref{thm:toric}.
\begin{proof}[{\bf Proof of Theorem \ref{thm:toric}}]
Let \[\begin{aligned}U^{h,\epsilon}&=(F^\epsilon)^{-1}\left(\left\{(h,v)\in\R^2\mid |v|\le \left(\frac{h}{2}\right)^{2/p}\right\}\right),\\
U^{h,\epsilon}_{int}&=(F^\epsilon)^{-1}\left(\left\{(h,v)\in\R^2\setminus\{0\}\mid |v|< \left(\frac{h}{2}\right)^{2/p}\right\}\right).\end{aligned}\] It is clear that $F(U^{h,\epsilon}_{int})$ is simply-connected. It follows from Theorem \ref{thm:al} that $U^{h,\epsilon}_{int}$ is symplectomorphic to $\mu^{-1}(\phi(U^{h,\epsilon}_{int}))$, where $\varphi$ are the action coordinates defined in \eqref{eq:phi}. We now define such $\phi$ using appropriate sets of curves and a Liouville form $\lambda$.
Fix $c=(h,v)$, and let
$r_{min},r_{max}$ be the smallest and largest solutions of~\eqref{eq-v-sqaure}. 
Set
\begin{align*}\yy_0&=(r_{max},0).\\
\xx_0&=(0,\text{sign}(v)\cdot r_{min}).\\
\end{align*}
Let $\sigma_0$ be the curve parametrized by $\varphi_t(\xx_0,\yy_0)$ with $t\in[0,t_0]$, where $\{\varphi_t\}$ is the flow of the vector field $X_{H_p^\epsilon}$, and  $t_0$ is the smallest $t>0$ such that the norm of the $\yy$-component of $\varphi_t(\xx_0,\yy_0)$ is $r_{max}$. Let $(\widetilde{\xx}_0,\widetilde{\yy}_0)=\varphi_{t_0}(\xx_0,\yy_0)$. Note that there exists $A\in SO(2)$ such that $A\widetilde{\xx}_0=\xx_0$ and $A\widetilde{\yy}_0=\yy_0$. Let $A_1^t$ and $A_2^t$ be two simple curves in $SO(2)\cong S^1$ connecting the identity with $A$, and rotating counterclockwise and clockwise, respectively. For $i=1,2$, let $\sigma_i$ denote the curve parametrized by $(A_i^t\widetilde{\xx}_0,A_i^t\widetilde{\yy}_0)$, and 
define $\gamma_i^c$ to be the composition of the curve $\sigma_0$ with $\sigma_i$. We observe that by definition $\{\gamma_1^c,\gamma_2^c\}$ generates $H_1\left((F^\epsilon)^{-1}(c)\right)$. 

\medskip

Next, let $\lambda=\sum_{i=1}^2 x_idy_i$. We write $(\xx,\yy)$ in polar coordinates $(p_r,p_\theta,r,\theta)$ as in the proof of Proposition \ref{prop:properties-of-Hamiltonian-l_p}.
If $v=0$, then $\lambda=p_r\,dr$. Hence, in this case
\begin{equation}\label{eq:gamma_i0}
\int_{\gamma_i^c}\lambda=\int_{\gamma_i^c}p_r\,dr=\int_{\sigma_0}p_r\,dr.
\end{equation}
Now suppose that $v\neq 0$. So $\gamma_i^c$ does not go through the origin $\yy=0$. For $i=1,2$,
\begin{equation}
\int_{\gamma_i^c}\lambda=\int_{\gamma_i^c}p_r\,dr+\int_{\gamma_i^c}p_\theta\,d\theta=\int_{\sigma_0}p_r\,dr+v\int_{\gamma_i^c}d\theta.\label{eq:gamma_i}
\end{equation}
For all $v$, we can compute the integral
\begin{equation}
\int_{\sigma_0}p_r\,dr=2\int_{r_{min}}^{r_{max}}\left(h^{2/p}f^\epsilon(h^{-2/p}r^2)-\frac{v^2}{r^2}\right)^{1/2}\,dr.\label{eq:sigma0}
\end{equation}
Let $g_p^\epsilon(h,v)$ be the function defined by the expression in \eqref{eq:sigma0}. We also observe that
\begin{equation}\label{eq:theta}
\begin{aligned}
\int_{\gamma_1^c}d\theta&=\left\{\begin{aligned} \;\;\; 2\pi &,\text{ if }v>0,\\
0 &,\text{ if }v<0.
\end{aligned}\right.\\
\int_{\gamma_2^c}d\theta&=\left\{\begin{aligned} 0 &,\text{ if }v>0,\\
-2\pi &,\text{ if }v<0.
\end{aligned}\right.
\end{aligned}
\end{equation}
It follows from \eqref{eq:gamma_i0}, \eqref{eq:gamma_i}, \eqref{eq:sigma0} and \eqref{eq:theta} that
\begin{equation}\label{eq:phieps}
\phi^\epsilon(h,v)=\left\{\begin{aligned}
(g_p^\epsilon(h,v)+2\pi v,g_p^\epsilon(h,v)),&\text{ if }v\ge 0,\\
(g_p^\epsilon(h,v),g_p^\epsilon(h,v)-2\pi v),&\text{ if }v<0.
\end{aligned}\right.
\end{equation}

It is easy to see that $\phi^\epsilon$ extends to a function defined on $U^{h,\epsilon}$. Finally, to see that the symplectomorphism $U^{h,\epsilon}_{int}\cong\mu^{-1}(\phi^\epsilon(U^{h,\epsilon}_{int}))$ extends to a symplectomorphism $U^{h,\epsilon}\cong\mu^{-1}(\phi^\epsilon(U^{h,\epsilon}))$ we use a similar method to the one in \cite[Lemma 35]{ramossepe}. Namely, we first use a theorem of Eliasson~\cite{eliasson} to show that the symplectomorphism extends to the pre-images of the points $(h,v)\neq(0,0)$ such that $|v|=\left({h}/{2}\right)^{2/p}$, see \cite{duf_mol,eliasson}. Next, to extend the symplectomorphism to $(0,0)$, we use a Theorem of Gromov-McDuff \cite[Theorem 9.4.2]{ms} as it was done in the proof of \cite[Theorem 3]{bidisk}.

\medskip

Next we observe that by definition $f^{\epsilon}(t)\to (1-t^{p/2})^{2/p}$ as $\epsilon \to 0$, and hence $h^{2/p}f^{\epsilon}(h^{-2/p}r^2)\to (1-r^p)^{2/p}$ as $(h,\epsilon)\to (1,0)$. Moreover, $r_{min}$ and $r_{max}$ converge to the two roots of the equation $r^2(1-r^p)^{2/p}=v^2$, namely
\[r_{\pm}=\left(\frac{1}{2}\pm\sqrt{\frac{1}{4}-v^p}\right)^{1/p}.\] It follows from \eqref{eq:g}, \eqref{eq:sigma0} and \eqref{eq:phieps} that $\phi^\epsilon(h,v)$ converges to \eqref{eq:curve} as $(h,\epsilon)\to(1,0)$. Therefore $\bigcup_{\epsilon>0}U^{1-\epsilon,\epsilon}=X_{\Omega_p}$. Note that $X_{\Omega_p}$ is open, unlike $U^{1-\epsilon,\epsilon}$

Now let $\{\epsilon_n\}$ be a sequence such that $\epsilon_n\to 0$ as $n\to \infty$. 
We can then argue as in the proof of \cite[Theorem 3]{bidisk} to conclude that we can find possibly different symplectomorphisms $\Phi_n:U^{1-\epsilon_n,\epsilon_n}\to\phi^{\epsilon_n}(U^{1-\epsilon_n,\epsilon_n})$ for all $n$ such that \[\Phi_{n+1}|_{U^{1-\epsilon_n,\epsilon_n}}=\Phi_n.\] Finally, define $\Phi: {\mathbb X}_p \to \C^2$ by $\Phi(z)= \Phi_n(z)$ if $z\in U^{1-\epsilon_n,\epsilon_n}$. Therefore $\Phi$ is a symplectic embedding whose image is $X_{\Omega_p}$, and the proof of the theorem is complete.
\end{proof}

\subsection{Convexity/Concavity of the toric image} \label{section-toric-domain-of-Lagrangian-p-sum}

In this section we prove Proposition~\ref{prop:cc}. We start by establishing the following properties of the function $g_p$ which appear in the toric description~$(\ref{eq:curve})$.

\begin{lemma}\label{lem:properties-of-gp}
The function $g_p:[0,1/4^{1/p}]\to \R$ defined in~\eqref{eq:g} satisfies the following:
\begin{enumerate}[(a)]
\item $g_p(0)=A(p)/2$, and $g_p(1/4^{1/p})=0$, for every $ p \geq 1$.
\item $g_p$ is strictly decreasing for all $p \geq 1$.
\item $g_p$ is strictly concave if $1<p<2$, and strictly convex if $p>2$.
\item For $p=2$ one has $g_2(v)=\frac{\pi}{2}-\pi v$.
\item $\displaystyle\lim_{v\to 0} g_p'(v) =-\pi$ and 
 $\displaystyle\lim_{v\to (1/4)^{1/p}} g_p'(v)= -\sqrt{\tfrac {2}{p}} \pi$, for every $p \geq 1$.
\item If $p\ge 9/2$, the derivative $g_p'$ is injective and its image contains the point $-2\pi/3$.
\end{enumerate}
\end{lemma}

\begin{proof}[{\bf Proof of Lemma~\ref{lem:properties-of-gp}}]
Note first that the boundaries of the integral in~\eqref{eq:g} guarantee that the integrand is well defined.
Moreover, one can check directly that the function $g_p(v)$ is differentiable in the interval $(0,1/4^{1/p})$, for every $p \geq 1$.

\begin{enumerate}[(a)]
\item This follows immediatly from~\eqref{eq:area-p-norm} and~\eqref{eq:g}.
\item By differentiating~\eqref{eq:g}, one has that for every $p \geq 1$ and $v \in (0,1/4^{1/p})$
\begin{equation}\label{eq:dg}
g_p'(v)=-2\int_{\left(\frac{1}{2}-\sqrt{\frac{1}{4}-v^p}\, \right)^{1/p}}^{\left(\frac{1}{2}+\sqrt{\frac{1}{4}-v^p} \,\right)^{1/p}}\left((1-r^p)^{2/p}-\frac{v^2}{r^2}\right)^{-1/2}\frac{v}{r^2}\,dr.
\end{equation}
Hence $g_p'(v)<0$ for all $v\in(0,1/4^{1/p})$, and therefore $g_p$ is strictly decreasing.

\item We first change variables in~\eqref{eq:dg} by setting
$w=r^p-1/2$, and obtain
\begin{equation}\label{eq:dg2}
\begin{aligned}
g_p'(v)&=-\frac{2}{p}\int_{-\sqrt{\frac{1}{4}-v^p}}^{\sqrt{\frac{1}{4}-v^p}} 
{\frac {v} {\left ({w+\frac{1}{2}} \right) \sqrt{\left({\frac 1 4}-w^2 \right)^{2/p}-v^2}   }}
\,dw \\&
 =-\frac{2}{p}\int_{0}^{\sqrt{\frac{1}{4}-v^p}} 
{\frac {v} {\left ({\frac{1}{4}-w^2} \right) \sqrt{\left({\frac 1 4}-w^2 \right)^{2/p}-v^2}   }}
\,dw.
\end{aligned}
\end{equation}
Next we change variables again by setting
 \begin{equation}\label{eq:sub1}x=\frac{v^{p/2}w}{\sqrt{\frac{1}{4}-v^p-w^2}}.\end{equation}
It follows from a straight-forward calculation that
\begin{align}\frac{dw}{dx}&=\frac{v^p\sqrt{\frac{1}{4}-v^p}}{\left (v^p+x^2 \right)^{3/2}}\label{eq:sub2},\\
\frac{1}{4}-w^2&=\frac{v^p}{v^p+x^2}\left(\frac{1}{4}+x^2\right)\label{eq:sub3}.
\end{align}
From \eqref{eq:dg2}, \eqref{eq:sub1}, \eqref{eq:sub2} and \eqref{eq:sub3} we obtain
\begin{equation}
\begin{aligned}
\label{eq:dg3}
g_p'(v)&=-\frac{2}{p}\int_0^\infty\frac{v\sqrt{\frac{1}{4}-v^p}}{\sqrt{v^p+x^2}\left(\frac{1}{4}+x^2\right)\sqrt{\left(\frac{v^p}{v^p+x^2}\left(\frac{1}{4}+x^2\right)\right)^{2/p}-v^2}}\,dx\\
&=-\frac{2}{p}\int_0^{\infty}\frac{\sqrt{\frac{1}{4}-v^p}}{\sqrt{v^p+x^2}\left(\frac{1}{4}+x^2\right)\sqrt{\left(\frac{\frac{1}{4}+x^2}{v^p+x^2}\right)^{2/p}-1}}\,dx\\
&=-\frac{2}{p}\int_0^\infty\frac{1}{\frac{1}{4}+x^2}\cdot\sqrt{\frac{\frac{\frac{1}{4}+x^2}{v^p+x^2}-1}{\left(\frac{\frac{1}{4}+x^2}{v^p+x^2}\right)^{2/p}-1}}\,dx.
\end{aligned}
\end{equation}
Note that for $u>1$, the function $u\mapsto \frac{u-1}{u^{2/p}-1}$ is strictly decreasing or increasing if $0<p<2$ or $p>2$, respectively. Moreover, for fixed $x$ the function $v\mapsto \frac{\frac{1}{4}+x^2}{v^p+x^2}$ is strictly decreasing. Therefore, $g_p'$ is strictly decreasing or increasing if $0<p<2$ or $p>2$, respectively, which proves the claim.
\item If we let $p=2$ in \eqref{eq:dg3} then,
\begin{equation*}
g_2'(v)=-\int_0^{\infty}\frac{1}{\frac{1}{4}+x^2}\,dx=-\pi.
\end{equation*}
From (a), it follows that $g_2(0)=\pi/2$, and hence $g_2(v)=\frac{\pi}{2}-\pi v$, as required.
\item
It follows from \eqref{eq:dg3} that
\[
\lim_{v\to 0} g_p'(v)=-\frac{1}{p}\int_0^\infty \frac{dx}{x\left(\frac{1}{4}+x^2\right)\sqrt{\left(\frac{1}{4x^2}+1\right)^{2/p}-1}}.\]
Setting $t=(\frac{1}{4x^2}+1)^{1/p}$, we obtain
\[
\lim_{v\to 0} g_p'(v)=-2\int_1^\infty \frac{dt}{t\sqrt{t^2-1}}=-2\arctan \sqrt{t^2-1}\, \Big|_1^{\infty}=-\pi.\]
Next, observe that
\begin{equation}\label{eq:lim}
\begin{aligned}
\lim_{v\to(1/4)^{1/p}} \frac{\left(\frac{1}{4}+x^2\right)^{2/p}-(v^p+x^2)^{2/p}}{\frac{1}{4}-v^p}&=
\frac{2}{p}\left(\frac{1}{4}+x^2\right)^{2/p-1}.
\end{aligned}
\end{equation}
Relations \eqref{eq:dg3} and \eqref{eq:lim} yield
\[
\begin{aligned}
\lim_{v\to(1/4)^{1/ p}} g_p'(v) &=
-\frac{2}{p}\int_0^\infty \Bigl(\frac{1}{4}+x^2\Bigr)^{{\frac 1 p}-{\frac 3 2}}\sqrt{\lim_{v\to(1/4)^{1/p}}\frac{\frac{1}{4}-v^p}{\left(\frac{1}{4}+x^2\right)^{2/p}-(v^p+x^2)^{2/p}} }\,dx\\
&=-\frac{2}{p}\cdot\sqrt{\frac{p}{2}}\int_0^\infty\left(\frac{1}{4}+x^2\right)^{-1}\,dx
=-\sqrt{\frac{2}{p}}\pi.\end{aligned}\]
\item We proved in (c) that $g_p'$ is strictly increasing if $p>2$. In particular, $g_p'$ is injective if $p>2$. It follows from (e) that the image of $g_p'$ is the interval $\left[-\sqrt{\frac{2}{p}}\pi,-\pi\right]$, which contains the point $-2\pi/3$ if $p\ge 9/2$.
\end{enumerate}
\end{proof}


\begin{proof}[{\bf Proof of Proposition \ref{prop:cc}}]
Let $p\ge 1$, and let $(x(v),y(v))$ be the parametrization of the curve given by \eqref{eq:curve}. It follows from Proposition \ref{lem:properties-of-gp} that $g_p'(v)\ge -\sqrt{2} \pi$ for $v\ge 0$. Thus, $x'(v)>0$ for all $v$, and hence this curve is the graph of a decreasing function $\varphi$. Moreover, by definition one has
\[y''(v)x'(v)-x''(v)y'(v)=2\pi g_p''(|v|).\]
Thus, it follows from Lemma \ref{lem:properties-of-gp} (c),(d) that $X_{\Omega_p}$ is a concave toric domain if $1 \le p\le 2$, and a convex toric domain if $p\ge 2$.
\end{proof}

\section{The ECH capacities of toric domains} \label{sec:cap}

In \cite{qech}, Hutchings defined a sequence of symplectic capacities for 4-dimensional symplectic manifolds using embedded contact homology (ECH). In particular, for a Liouville domain $X\subset \R^4$, he defined a sequence of numbers $\left(c_k(X)\right)_{k\in\mathbb{N}}\subset\R\cup\{\infty\}$ satisfying:
\begin{itemize}
\item $0=c_0(X)\le c_1(X)\le c_2(X)\le \dots\le\ \infty$,
\item $c_k(a\cdot X)=a^2\cdot c_k(X)$, for all $k\in\mathbb{N}$ and $a>0$,
\item $X_1\hookrightarrow X_2\Rightarrow c_k(X_1)\le c_k(X_2)$, for all $k\in\mathbb{N}$. 
\item $\left(c_k(B(a))\right)_{k\in\mathbb{N}}=(0,a,a,2a,2a,2a,3a,3a,3a,3a,\dots)$.
\end{itemize}
\medskip
The ECH capacities turn out to give sharp obstructions for many symplectic embedding problems (see e.g.,~\cite{McD}). Moreover, for convex and concave toric domains, they can be computed combinatorially as explained in \cite{qech,ccfhr,hutbey}.
We will now review some relevant properties of the ECH capacities, and in particular describe the first two capacities of symmetric concave/convex toric domains.

\medskip

Let $X_\Omega$ be a concave toric domain. The weight expansion $w(\Omega)$, associated with $X_{\Omega}$, is a multiset which was defined inductively in \cite{ccfhr} as follows. Let $T(c)\subset\R^2$ be the triangle whose vertices are $(0,0)$, $(c,0)$ and $(0,c)$. For a set $\Omega \subset\R_{\ge 0}^2$ which is bounded by the coordinate axes and the graph of a decreasing concave function $\phi:[0,a]\to\R_{\ge 0}$ with $\phi(a)=0$, we define
\begin{equation} \label{eq-def-of-tau} \tau(\Omega) := \sup \{ c \, | \, T(c) \subseteq \Omega \}. \end{equation} 
We write $\Omega\setminus T(\tau(\Omega))=\widetilde{\Omega}_1\sqcup\widetilde{\Omega}_2$, where $\widetilde \Omega_1$ does not intersect the $y$-axis and $\widetilde \Omega_2$ does not intersect the $x$-axis. Note that these sets can be empty, otherwise their closures have a unique obtuse corner. Consider the closures of $\widetilde{\Omega}_1$ and $\widetilde{\Omega}_2$ translated so that each obtuse corner is mapped to the origin. Let $\Omega_1$ and $\Omega_2$ be the images of these translations under multiplication by the matrices 
\[
\begin{pmatrix}
    1  &  1      \\
    0  &  1      
\end{pmatrix}
\ {\rm and} \  
\begin{pmatrix}
    1  &  0      \\
    1  &  1      
\end{pmatrix}, 
\]
respectively. Note that $X_{\Omega_1}$ and $X_{\Omega_2}$ are again concave toric domains if non-empty. The weight expansion is defined inductively by 
\[w(\Omega)=\{\tau(\Omega)\}\sqcup w(\Omega_1)\sqcup w(\Omega_2),\] where this union is considered with repetition, and $w(\Omega_j)=\emptyset$ if $\Omega_j = \emptyset$. 
Another way of seeing it is the following. The process above defines a directed tree of domains starting from $\Omega$. We denote the elements of this tree by $\Omega_{i_1\dots i_p}$, where the indices $i_1,\dots,i_p\in\{1,2\}$, and the domains derived from $\Omega_{i_1\dots i_p}$ are $\Omega_{i_1\dots i_p 1}$ and $\Omega_{i_1\dots i_p 2}$. It is possible that some $\Omega_{i_1\dots i_p}$ are empty. Then, \[w(\Omega)=\{\tau(\Omega_{i_1\dots i_p}) \mid p\in\mathbb{N}; \ i_1,\dots,i_p\in\{1,2\}\}.\]
With a slight abuse of notation, we now write $w(\Omega)=(w_1,w_2,w_3,\dots)$, where
\begin{equation}\label{eq:dec}
w_1\ge w_2\ge w_3\ge \cdots\end{equation} Note that $w_1=\tau(\Omega)$ and that $w_2=\max(\tau(\Omega_1),\tau(\Omega_2))$.
It was shown in \cite{ccfhr} that for a concave toric domain $X_{\Omega}$ one has
\begin{equation}\label{eq:ck}c_k(X_\Omega)=c_k\left(\bigsqcup_{j=1}^{\infty} B(w_j)\right)=\max_{i_1+\dots+i_k=k}\sum_{j=1}^{k} c_{i_j}(B(w_j)).\end{equation}

Next, we say that a toric domain $X_\Omega$ is \textit{symmetric} if it is invariant under the reflection about the line $y=x$. The following lemma is a computation of the first two ECH capacities that will be relevant for all of the domains in this paper.

\begin{lemma}\label{lemma:c1c2}
Let $X_\Omega$ be a symmetric toric domain where $\Omega\subset \R^2_{\ge 0}$ is bounded by the coordinate axes and a $C^1$-curve $\gamma$ parametrized by $(x(v),y(v))$, which connects the points $(a,0)$ and $(0,a)$. Let $b\in(0,a)$ such that $(b,b)=(x(v),y(v))$ for some $v$.
\begin{enumerate}[(a)]
\item If $X_\Omega$ is a convex toric domain, then 
\begin{align*}
c_1(X_\Omega)&=a,\\
c_2(X_\Omega)&=2b.
\end{align*}
\item If $X_\Omega$ is a concave toric domain such that $-2\le y'(v)/x'(v)\le -1/2$, then
\begin{align*}
c_1(X_\Omega)&=2b,\\
c_2(X_\Omega)&=a.
\end{align*}
\item If $X_\Omega$ is a concave toric domain such that $y'(v)/x'(v)>-1/2$ for some $v$, and let $v_0$ such that $y'(v_0)/x'(v_0)=-1/2$. Then
\begin{align*}
c_1(X_\Omega)&=2b,\\
c_2(X_\Omega)&=2y(v_0)+x(v_0).
\end{align*}
\end{enumerate}
\end{lemma}

\begin{proof}[{\bf Proof of Lemma~\ref{lemma:c1c2}}]

(a) If $X_\Omega$ is a convex toric domain bounded by the coordinate axes and a $C^1$ curve $\gamma$ connecting $(x_0,0)$ to $(0,y_0)$, then it follows from \cite[Proposition 5.6]{hutbey} that
\[c_1(X_\Omega)=\min(x_0,y_0),\quad c_2(X_\Omega)=\min(2x_0,2y_0,x_1+y_1)=x_1+y_1,\]
where $(x_1,y_1)$ is a point on the curve at which the slope of the tangent line is $-1$. Since we  assume that $\gamma$ is symmetric about the line $y=x$, we conclude that
\[c_1(X_\Omega)=a,\quad c_2(X_\Omega)=2b.\]

\noindent We turn now to prove (b) and (c). Suppose that $X_\Omega$ is a concave toric domain and let $(w_1,w_2,\dots)$ be its weight expansion. 
It follows from the \eqref{eq:dec}, \eqref{eq:ck} and the computation of $c_k(B(c))$ that
\begin{equation}\label{eq:c1c2}
c_1(X_\Omega)=w_1,\quad c_2(X_\Omega)=w_1+w_2.
\end{equation}

Since $\gamma$ is a $C^1$ curve, $w_1$ is the unique real number such that the line $x+y=w_1$ is tangent to $\gamma$. Since this curve is symmetric about the line $y=x$, it follows that $c_1(X_\Omega)=w_1=2b$. Let $x_{1/2}$ be the $x$-intercept of the line of slope $-1/2$ whose intersection with the first quadrant is as large as possible but still contained in $\Omega$. Simple linear algebra shows that $w_2=x_{1/2}-w_1$. So if $-2\le y'(v)/x'(v)\le -1/2$ for all $v$, then 
$x_{1/2}=a$, and from \eqref{eq:c1c2} we obtain
\[c_2(X_\Omega)=w_1+(a-w_1)=a.\]
On the other hand, suppose that $y'(v)/x'(v)>-1/2$ for some $v$ (which implies that $x_{1/2}<a$), and let $v_0$ 
such that $y'(v_0)/x'(v_0)=-1/2$. In this case the upper right side of the triangle $T'(w_2)+(w_1,0)$ is tangent to $\gamma$. So the point $(w_2+w_1,0)$ belongs to the line of slope $-1/2$ going through $(x(v_0),y(v_0))$. Therefore, in this case
\begin{equation}\label{eq:w1w2}
c_2(X_\Omega)=w_1+w_2=2y(v_0)+x(v_0).\end{equation}
This completes the proof of Lemma~\ref{lemma:c1c2}.
\end{proof}

The first two ECH capacities of the toric domain $X_{\Omega_p}$ introduced in Theorem~\ref{thm:toric}  are now a straight forward consequence of Lemma \ref{lem:properties-of-gp} and Lemma \ref{lemma:c1c2}. 

\begin{prop}\label{prop:cap}
The first two ECH capacities of $X_{\Omega_p}$ are given as follows.
\quad
\begin{enumerate}[(a)]
\item If $p \in[1,2]$, then 
\begin{equation} \label{eq-first-cap-of-lagrangian-l-p}
\begin{aligned}
c_1(X_{\Omega_p})&=2\pi(1/4)^{1/p},\\c_2(X_{\Omega_p})&= A(p).
\end{aligned}
\end{equation}
\item If $p\ge 2$, then 
\begin{equation} \label{eq-second-cap-of-lagrangian-l-p}
\begin{aligned}
c_1(X_{\Omega_p})&=A(p)\\
c_2(X_{\Omega_p})&=\left\{
\begin{array}{ll}
2\pi(1/4)^{1/p},&\text{if }p\le 9/2,\\
2\pi \left(g_p'\right)^{-1}(-2\pi/3)+3g_p\left(\left(g_p'\right)^{-1}(-2\pi/3)\right),&\text{if }p>9/2.
\end{array}\right.
\end{aligned}
\end{equation}
\end{enumerate}
\end{prop}

\subsection{Ball packings and symplectic embeddings}

In this section we provide a criterion for the embedding of a symmetric concave domain into a ball in ${\mathbb R}^4$ (see Proposition~\ref{prop:flex} below). We start with recalling the following result proved by Cristofaro-Gardiner in~\cite{cgcc}.

\begin{thm}\label{thm:cgcc}
Let $X_{\Omega}$ be a concave toric domain with weight expansion $(w_1,w_2,\dots)$, and let $X_{\Omega'}$ be a convex toric domain. Then $X_{\Omega}\hookrightarrow X_{\Omega'}$ if and only if
\[\bigsqcup_{i=1}^N B(w_i)\hookrightarrow X_{\Omega'},\quad\forall N.\]
\end{thm}
Next we recall a criterion for the existence of a ball packing
\begin{equation}\label{eq:bp}\bigsqcup_{i=1}^N B(a_i)\hookrightarrow B(c),\end{equation}
based on the so called ``Cremona transformations''.
From now on we assume that $N\ge 3$. A vector $(c;a_1,a_2,\dots,a_N)\in\R^{N+1}$ is called \textit{ordered} if $c\ge a_1\ge a_2\ge \cdots$. An ordered vector is called \textit{reduced} if $c\ge a_1+a_2+a_3$ and $a_i\ge 0$ for all $i$. The Cremona transform is the linear transformation  
\[(c;a_1,a_2,\dots,a_N)\mapsto (2c-a_1-a_2-a_3;c-a_2-a_3,c-a_1-a_3,c-a_1-a_2,a_4,\dots,a_N).\]
A \textit{Cremona move} takes a vector $(c;a_1,a_2,\dots,a_N)$ to the vector obtained by ordering its image under the Cremona transform. The following theorem is a combination of results in \cite{mcduffblowup,mcdpolt,lili}, as explained for example in \cite{busepin}.
\begin{thm}\label{thm:cremona}
Let $(c;a_1,a_2,\dots,a_N)$ be an ordered vector, and $N \geq 3$. Then there exists a symplectic embedding 
\[\bigsqcup_{i=1}^N B(a_i)\hookrightarrow B(c)\]
if and only if the vector $(c;a_1,a_2,\dots,a_N)$ is taken to a reduced vector under a finite number of Cremona moves, and
\[\sum_{i=1}^N a_i^2\le c^2.\]
\end{thm}

Now suppose that $X_\Omega$ is a symmetric concave toric domain. Its weight sequence is of the form $(w_1,w_2,w_2,w_3,w_3,\dots)$. As an application of the two theorems above we prove the following result, which will be used later and may be also of independent interest.

\begin{prop}\label{prop:flex}
Let $X_\Omega$ be a symmetric concave toric domain with weight sequence $(w_1,w_2,w_2,w_3,w_3,\dots)$, and let $c=c_2(X_\Omega)$. Suppose that $\text{vol}(X_\Omega)\le \text{vol}(B(c))$, and 
\begin{equation}\label{eq:w2w3w4}
\tau(\Omega_1)\ge \tau(\Omega_{11})+\tau(\Omega_{111}),
\end{equation}
where $\tau(\cdot)$ is defined in \eqref{eq-def-of-tau}.
Then there exists a symplectic embedding
\[X_\Omega\hookrightarrow B(c).\]
\end{prop}
\begin{proof}[{\bf Proof of Proposition~\ref{prop:flex}}]
We assume that $X_\Omega$ is not a ball, otherwise the result is trivial. Thus, $w_2\neq 0$ and it follows from \eqref{eq:w1w2} that $c_2(X_{\Omega})=w_1+w_2$. By Theorems \ref{thm:cgcc} and \ref{thm:cremona}, it suffices to show that for every $N \geq 2$, the vector
\begin{equation}\label{eq:vec1}(w_1+w_2;w_1,w_2,w_2,w_3,w_3,w_4,w_4,\dots,w_N,w_N)\end{equation}
can be turned into a reduced vector after a finite number of Cremona moves. We remark that we are always assuming that a vector of this form is ordered. After one Cremona move, \eqref{eq:vec1} becomes
\begin{equation}\label{eq:vec2}(w_1;w_1-w_2,w_3,w_3,w_4,w_4\dots,w_N,w_N,0,0).\end{equation}
We need to check that $w_1-w_2\ge w_3$. Note that $w_2=\tau(\Omega_1)=\tau(\Omega_2)$. Since $\Omega$ is symmetric, it follows that the height of $\Omega_1$ is $w_1/2$. Hence $w_2\le w_1/2$, and consequently \[w_1-w_2\ge \frac{w_1}{2}\ge w_2\ge w_3.\]
If $w_3=0$, then \eqref{eq:vec2} is reduced, so we assume that $w_3\neq 0$.
We claim that for $k \geq 2$
\begin{equation}\label{eq:cond}w_2< w_{k}+w_{k+1}\Rightarrow w_2+\dots+w_k\le \frac{w_1}{2}.\end{equation} Note that \eqref{eq:cond} holds for $k=2$ since $w_2\le w_1/2$. Now suppose that \begin{equation}\label{eq:w2wk}w_2< w_{k}+w_{k+1}\end{equation} for some $k>2$. We can write $w_k=\tau(\Omega_{1 i_1\dots i_p})$ and $w_{k+1}=\tau(\Omega_{1 j_1\dots j_l})$, where $p,l\neq 0$. First suppose that $i_q\neq j_q$ for some $q$, and let $r$ be the smallest such $q$. Then, 
\begin{equation}\label{eq:tau2}\tau(\Omega_{1 i_1\dots i_p})+\tau(\Omega_{1 j_1\dots j_l})\le \tau(\Omega_{1 i_1\dots i_r 1})+\tau(\Omega_{1 i_1\dots i_r 2}).\end{equation}
We now observe (see Figure \ref{fig:tor}(a)) that \begin{equation}\label{eq:tau}\tau(\Omega_{1 i_1\dots i_r 1})+\tau(\Omega_{1 i_1\dots i_r 2})\le \tau(\Omega_{1 i_1\dots i_r}).\end{equation} This is because $\tau(\Omega_{1 i_1\dots i_r 1})$ and $\tau(\Omega_{1 i_1\dots i_r 2})$ are the $y$-coordinate and $x$-coordinate of the intersections between the line $x+y=\tau(\Omega_{1 i_1\dots i_r})$ and the lines of slope $-1/2$ and $-2$ which intersect $\partial \Omega_{1 i_1\dots i_r}\cap \R^2_{>0}$ and not $\R^2_{>0}\setminus\Omega_{1 i_1\dots i_r}$, respectively. The latter condition mean that these lines are tangent to $\partial \Omega_{1 i_1\dots i_r}\cap \R^2_{>0}$ if the function defining this domain is $C^1$.
It follows from \eqref{eq:tau2} and \eqref{eq:tau} that
\[w_{k}+w_{k+1}=\tau(\Omega_{1 i_1\dots i_p})+\tau(\Omega_{1 j_1\dots j_l})\le \tau(\Omega_{1 i_1\dots i_r})\le \tau(\Omega_{1 })=w_2,\]
which contradicts \eqref{eq:w2wk}.
\begin{figure}
\begin{subfigure}{0.5\textwidth}
\begingroup%
  \makeatletter%
  \providecommand\color[2][]{%
    \errmessage{(Inkscape) Color is used for the text in Inkscape, but the package 'color.sty' is not loaded}%
    \renewcommand\color[2][]{}%
  }%
  \providecommand\transparent[1]{%
    \errmessage{(Inkscape) Transparency is used (non-zero) for the text in Inkscape, but the package 'transparent.sty' is not loaded}%
    \renewcommand\transparent[1]{}%
  }%
  \providecommand\rotatebox[2]{#2}%
  \newcommand*\fsize{\dimexpr\f@size pt\relax}%
  \newcommand*\lineheight[1]{\fontsize{\fsize}{#1\fsize}\selectfont}%
  \ifx\svgwidth\undefined%
    \setlength{\unitlength}{258.75bp}%
    \ifx\svgscale\undefined%
      \relax%
    \else%
      \setlength{\unitlength}{\unitlength * \real{\svgscale}}%
    \fi%
  \else%
    \setlength{\unitlength}{\svgwidth}%
  \fi%
  \global\let\svgwidth\undefined%
  \global\let\svgscale\undefined%
  \makeatother%
  \begin{picture}(1,0.69565213)%
    \lineheight{1}%
    \setlength\tabcolsep{0pt}%
    \put(0,0){\includegraphics[width=\unitlength,page=1]{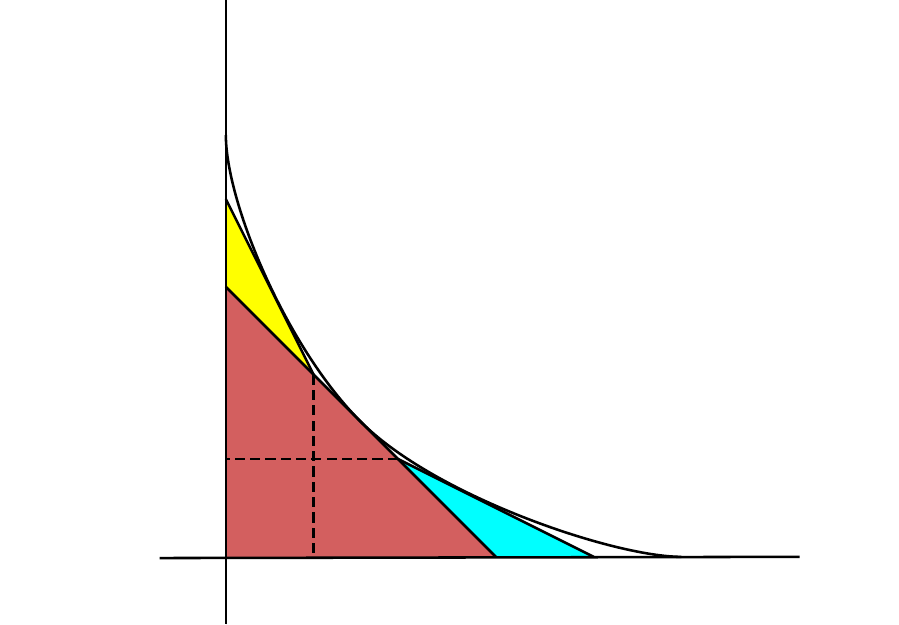}}%
    \put(0.26312222,0.01566838){\color[rgb]{0,0,0}\makebox(0,0)[lt]{\lineheight{1.25}\smash{\begin{tabular}[t]{l}$\tau(\Omega_{1 i_1\dots i_r 2})$\end{tabular}}}}%
    \put(0,0){\includegraphics[width=\unitlength,page=2]{w1.pdf}}%
    \put(0.00236301,0.13276615){\color[rgb]{0,0,0}\makebox(0,0)[lt]{\lineheight{1.25}\smash{\begin{tabular}[t]{l}$\tau(\Omega_{1 i_1\dots i_r 1})$\end{tabular}}}}%
    \put(0,0){\includegraphics[width=\unitlength,page=3]{w1.pdf}}%
    \put(0.65483455,0.2251301){\color[rgb]{0,0,0}\makebox(0,0)[lt]{\lineheight{1.25}\smash{\begin{tabular}[t]{l}$\Omega_{1 i_1\dots i_r 1}$\end{tabular}}}}%
    \put(0.38163164,0.49837592){\color[rgb]{0,0,0}\makebox(0,0)[lt]{\lineheight{1.25}\smash{\begin{tabular}[t]{l}$\Omega_{1 i_1\dots i_r 2}$\end{tabular}}}}%
    \put(0.41746015,0.36970454){\color[rgb]{0,0,0}\makebox(0,0)[lt]{\lineheight{1.25}\smash{\begin{tabular}[t]{l}$\Omega_{1 i_1\dots i_r}$\end{tabular}}}}%
  \end{picture}%
\endgroup%

\caption{}
\end{subfigure}
\begin{subfigure}{0.5\textwidth}
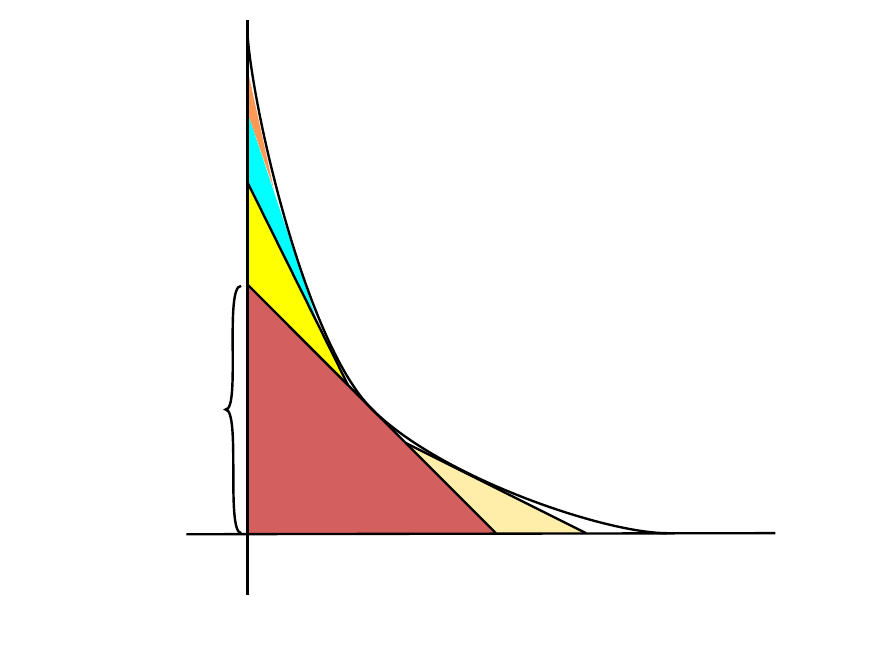
\caption{}
\end{subfigure}
\caption{}
\label{fig:tor}
\end{figure}
Now suppose that $i_q=j_q$ for all $1\le q\le \min(p,l)$. Note that $p<l$ since $w_k\ge w_{k+1}$. Therefore,
\begin{equation}\label{eq:tau3}
\tau(\Omega_{1 i_1\dots i_p})+\tau(\Omega_{1 j_1\dots j_l})\le \tau(\Omega_{1 i_1\dots i_p})+\tau(\Omega_{1 i_1\dots i_p i_{p+1}}).
\end{equation}
Let us first assume that $i_{p+1}=1$. If $i_q=2$ for some $1\le q\le p$, then  
\begin{equation}\label{eq:tau4}\tau(\Omega_{1 i_1\dots i_p})+\tau(\Omega_{1 i_1\dots i_p i_{p+1}})\le \tau(\Omega_{1 i_1\dots i_{q-1} 2})\le \tau(\Omega_{1})=w_2.\end{equation}
Combining \eqref{eq:tau3} and \eqref{eq:tau4}, we again reach a contradiction with \eqref{eq:w2wk}. So the only way to satisfy \eqref{eq:w2wk} in this case is if $i_q=1$ for all $q$. Now let us assume that $i_{p+1}=2$. Similarly, if $i_q=1$ for some $1\le q\le p$, then 
\begin{equation}\label{eq:tau5}\tau(\Omega_{1 i_1\dots i_p})+\tau(\Omega_{1 i_1\dots i_p i_{p+1}})\le \tau(\Omega_{1 i_1\dots i_{q-1} 1})\le \tau(\Omega_{1})=w_2,\end{equation}
and again we reach a contradiction. Thus, to satisfy \eqref{eq:w2wk} it is necessary that $i_q=2$ for all $q$. We note that if $l<p$, then 
the only way to satisfy \eqref{eq:w2wk} is for one of the following two conditions to hold:
\begin{align}
w_{k}+w_{k+1}&= \tau(\Omega_{1 1\dots 1})+\tau(\Omega_{1 1\dots 1 1 \dots j_l})\label{eq:tau6}\\
w_{k}+w_{k+1}&= \tau(\Omega_{1 2\dots 2})+\tau(\Omega_{1 2 \dots 2 2 \dots j_l})\label{eq:tau7}.
\end{align}
If \eqref{eq:tau6} holds, then 
\[w_2< \tau(\Omega_{1 1})+\tau(\Omega_{1 1 1}),\]
which contradicts \eqref{eq:w2w3w4}. Finally, if \eqref{eq:tau7} hold, then for all $j\le k$, $w_j$ is of the form $\tau(\Omega_{1 2\dots 2})$, otherwise we would get a contradiction to \eqref{eq:w2wk}. So
\[w_j=\tau(\Omega_{1 \underbrace{\scriptstyle 2\dots 2}_{j-2}}).\]
Therefore, \[w_2+\dots+w_k\le \frac{w_1}{2},\]
see Figure \ref{fig:tor}(b). This completes the proof of \eqref{eq:cond}. Next, let \[m=\max\{n\mid w_2< w_n+w_{n+1}\}.\] Since $w_3\neq 0$, it follows that $m\ge 2$. We claim that for $2\le k\le m$, after $k-2$ Cremona moves \eqref{eq:vec2} turns into a vector of the form
\begin{equation}\label{eq:vec3}
(w_1+kw_2-2(w_2+\dots+w_k);w_1+(k-1)w_2-2(w_2+\dots+w_k),w_{k+1},w_{k+1},\dots,0).\end{equation}

We will prove this claim by induction on $k$. First observe that for $k=2$, \eqref{eq:vec3} is simply \eqref{eq:vec2}. Now suppose that \eqref{eq:vec3} holds for some $k<m$. Applying one Cremona move, we obtain
\begin{equation}\label{eq:vec4}
\begin{aligned}
(&w_1+(k+1)w_2-2(w_2+\dots+w_k+w_{k+1}); w_1+kw_2-2(w_2+\dots+w_k+w_{k+1}),\\&w_{k+2},w_{k+2},w_2-w_{k+1},w_2-w_{k+1},\dots,0).\end{aligned}\end{equation}
We need to check that \eqref{eq:vec4} is ordered. Note that since $k<m$, it follows that $w_{k+2}> w_2-w_{k+1}$. Moreover, from \eqref{eq:cond} we conclude that
\[ w_1+kw_2-2(w_2+\dots+w_k+w_{k+1})- w_{k+2}\ge 0.\]
So \eqref{eq:vec4} is ordered. Moreover all of its components are non-negative.
In particular after $(m-2)$ Cremona moves we obtain a non-negative vector
\begin{equation}\label{eq:vec5}
(w_1+mw_2-2(w_2+\dots+w_m);w_1+(m-1)w_2-2(w_2+\dots+w_m),w_{m+1},w_{m+1},\dots,0).\end{equation}
If $2w_{m+1}\le w_2$, then \eqref{eq:vec5} is reduced. If not, we apply one more Cremona move to obtain
\begin{equation}\label{eq:vec6}
\begin{aligned}
(&w_1+(m+1)w_2-2(w_2+\dots+w_{m+1});w_1+mw_2-2(w_2+\dots+w_{m+1}),\\&w_2-w_{m+1},w_2-w_{m+1},\dots,0).
\end{aligned}
\end{equation}
Note that \eqref{eq:vec6} is ordered since $w_2-w_{m+1}\ge w_{m+2}$. Since $w_2<2w_{m+1}$, it follows that \eqref{eq:vec6} is reduced. Hence, we conclude that \eqref{eq:vec1} can be turned into a reduced vector after a finite number of Cremona moves. Therefore by Theorems \ref{thm:cgcc} and \ref{thm:cremona}, \[X_\Omega\hookrightarrow B(w_1+w_2),\]
and the proof of the proposition is complete.
\end{proof}

\section{Proof of the main results} \label{sec:proof-main-results}
In this section we prove Theorems~\ref{thm:main},~\ref{thm:rigid}, and~\ref{thm:symplp}, and also Proposition~\ref{prop:lpellip}. 

\subsection{Lagrangian $\ell_p$-sum : rigidity}
Here we prove Theorem \ref{thm:rigid} (a,b,c,d), and the corresponding parts of Theorem \ref{thm:main}.
\begin{proof}[{\bf Proof of Theorem~\ref{thm:rigid} (a,b,c,d)}]
Recall first that the first two ECH capacities of the Euclidean ball satisfy $c_1(B(r))=c_2(B(r))=r$.
Throughout the proof  we shall frequently use the fact  that the interior of ${\mathbb X}_p$ is symplectomorphic to the toric domain $X_{\Omega_p}$, as proved in Theorem~\ref{thm:toric} above. In particular, combining Theorem~\ref{thm:toric} and Proposition~\ref{prop:cap} shows that the first two ECH capacities of ${\mathbb X}_p$ are given by \eqref{eq-first-cap-of-lagrangian-l-p} and \eqref{eq-second-cap-of-lagrangian-l-p}. We now split the proof into four parts.

\begin{enumerate}[(a)]
\item Suppose that $p \in[1,2]$ and that $B(r)\hookrightarrow {\mathbb X}_p$. 
Then, by Theorem~\ref{thm:toric}, Proposition~\ref{prop:cap}, and the properties of the ECH capacities,
\begin{equation}\label{eq:r1}r\le c_1({\mathbb X}_p)=2\pi(1/4)^{1/p},\end{equation}
and hence $r_{S}({\mathbb X}_p) \le 2\pi(1/4)^{1/p}$. On the other hand, let $(\xx,\yy)\in B(r)$, where $r=2\pi(1/4)^{1/p}$. 
It follows from \eqref{eq:r1} and the H\"{o}lder inequality that
\begin{align*}\Vert \xx\Vert^{p}+\Vert \yy\Vert^p&\le\left(\Vert \xx\Vert^2+\Vert \yy\Vert^2\right)^{\frac{p}{2}} \cdot 2^{1-\frac{p}{2}}
< 2^{\frac{p}{2}}\cdot \left(\tfrac{1}{4}\right)^\frac{1}{2}\cdot 2\cdot 2^{-\frac{p}{2}}=1.
\end{align*}
Thus, one has that 
$B(r)\subset {\mathbb X}_p$ for $r=2\pi(1/4)^{1/p}$. We conclude that $r_{S}({\mathbb X}_p)=2\pi(1/4)^{1/p}$, and moreover that the symplectic embedding problem $B\overset{?}{\hookrightarrow} {\mathbb X}_p$ is both rigid and torically rigid for $1 \le p \le 2$.

\item Suppose now that $p \in[2,\infty)$ and that $B(r)\hookrightarrow {\mathbb X}_p$. Using again Theorem~\ref{thm:toric}, Proposition~\ref{prop:cap}, and the properties of the ECH capacities, one has in this case
\begin{equation*}
r\le c_1({\mathbb X}_p)=A(p),
\end{equation*}
where $A(p)$ is given by \eqref{eq:area-p-norm}. Hence $r_{S}({\mathbb X}_p) \le A(p)$.
Since by Proposition~\ref{prop:cc} $X_{\Omega_p}$ is a concave toric domain, it follows from Remark \ref{rmk:refl} that $T\subset\Omega_p$, where $T$ is the triangle bounded by the coordinate axes and the line $x+y=A(p)$. Thus we conclude that $B(A(p))\subset X_{\Omega_p}$. Therefore, $r_{S}({\mathbb X}_p)=A(p)$, and  the symplectic embedding problem $B\overset{?}{\hookrightarrow} {\mathbb X}_p$ is torically rigid.

\item Suppose that $p\in[1,2]$ and that ${\mathbb X}_p \hookrightarrow B(r)$.
Looking now at the second ECH capacitiy (see \eqref{eq-first-cap-of-lagrangian-l-p}), one has 
\begin{equation*}
A(p)=c_2({\mathbb X}_p)\le r.
\end{equation*}
Hence, $R_{S}({\mathbb X}_p)\ge A(p)$. On the other hand, let $(t,t)$ be the intersection of the curve \eqref{eq:curve} and the line $y=x$. It follows from Proposition~\ref{prop:cc} that  $X_{\Omega_p}$ is now a convex toric domain, and thus Remark \ref{rmk:refl} shows that $\Omega_\alpha\subset T$, where $T$ is the triangle bounded by the coordinate axes and the line $x+y=2t$. In particular, $X_T=B(2t)$. From \eqref{eq:area-p-norm},\eqref{eq:g}, and \eqref{eq:curve} we conclude that
\begin{equation}\label{eq:t}t=g_p(0) 
=2\int_0^1(1-r^p)^{1/p}\,dr = {\frac {A(p)} 2}.\end{equation}
So $X_{\Omega_p}\subset B(A(p))$. Therefore, $R_{S}({\mathbb X}_p)=A(p)$, and the symplectic embedding problem ${\mathbb X}_p \overset{?}{\hookrightarrow}B(r)$ is torically rigid.

\item Finally, suppose that $p \in[2,9/2]$ and that ${\mathbb X}_p \hookrightarrow B(r)$.
As in the previous case, from the second ECH capacity we conclude that
\begin{equation}\label{eq:r4}
2\pi(1/4)^{1/p}=c_2({\mathbb X}_p)\le r,
\end{equation}
which implies that $R_{S}({\mathbb X}_p)\ge 2\pi(1/4)^{1/p}$. On the other hand, 
if $(\xx,\yy)\in {\mathbb X}_p$, then by H\"{o}lder's inequality 
\begin{align*}\pi\left(\Vert \xx\Vert^2+\Vert \yy\Vert^2\right)&\le \pi\left(\left(\Vert \xx\Vert^p+\Vert\yy\Vert^p\right)^{2/\alpha}\cdot 2^{1-2/p} \right) < 2\pi\left(\tfrac{1}{4}\right)^{1/p},
\end{align*}
and so ${\mathbb X}_p \subseteq B\left(2\pi\left(\frac{1}{4}\right)^{1/p}\right)$. Therefore, $R_{S}({\mathbb X}_p)=2\pi\left(\frac{1}{4}\right)^{1/p}$, and  the  symplectic embedding problem ${\mathbb X}_p \overset{?}{\hookrightarrow}B(r)$ is both rigid and torically rigid.
\end{enumerate} 
This completes the proof of the first four parts of Theorem \ref{thm:rigid}, and of the corresponding parts of Theorem \ref{thm:main}.
\end{proof}


\subsection{Lagrangian $\ell_p$-sum : flexibility for $p>9/2$} \label{sec-non-rigidity}
In this section, we finish the proofs of Theorems \ref{thm:main} and \ref{thm:rigid}. More precisely, we will compute the outer radius $R_{S}({\mathbb X}_p)$ for $p>9/2$, and show that the corresponding symplectic embedding problem is non-rigid. Assume $p>9/2$. It follows from Theorem~\ref{thm:toric} and Proposition \ref{prop:cap} that
\begin{equation}\label{eq:c2-9:2}
c_2({\mathbb X}_p)=c_2(X_{\Omega_p})=2\pi \left(g_p'\right)^{-1}(-2\pi/3)+3g_p\left(\left(g_p'\right)^{-1}(-2\pi/3)\right).
\end{equation}
Recall from the proof of Proposition \ref{prop:cap} that this number is the $x$-intercept of a tangent line to the curve \eqref{eq:curve}. Moreover, the $x$-intercept of this curve is $2\pi(1/4)^{1/p}$ by Lemma~\ref{lem:properties-of-gp} (a). This implies that
\[
X_{\Omega_p}\subset B(c)\iff c\ge 2\pi(1/4)^{1/p}.
\]
The same calculation as in the proof of Theorem \ref{thm:rigid} (d) above shows that
\[{\mathbb X}_p\subset B(c)\iff c\ge 2\pi(1/4)^{1/p}.\]
 Since $X_{\Omega_p}$ is a concave toric domain, it follows that 
\[c_2({\mathbb X}_p)=c_2(X_{\Omega_p})<2\pi(1/4)^{1/p}.\]
Therefore the following proposition implies that ${\mathbb X}_p\overset{?}{\hookrightarrow} B$ is non-rigid, and moreover that $R_{S}({\mathbb X}_p)=c_2({\mathbb X}_p)$, as claimed in Theorems~\ref{thm:main} and~\ref{thm:rigid}.
\begin{prop}\label{prop:emb}
For $p> 9/2$, there is a symplectic embedding $X_{\Omega_p}\hookrightarrow B(c_2(X_{\Omega_p})).$
\end{prop}

\begin{proof}[{\bf Proof of Proposition~\ref{prop:emb}}]
Let $p>9/2$. To establish the required embedding, we verify the conditions of Proposition \ref{prop:flex} above.  We first claim that \[\text{Vol}(X_{\Omega_p})=\text{Vol}({\mathbb X}_p)\le \text{Vol}(B(c_2({\mathbb X}_p)))=\text{Vol}(B(c_2(X_{\Omega_p}))).\]
Indeed, since for $p<\tilde{p}$ one has ${\mathbb X}_p \subset {\mathbb X}_{\tilde p}$, it follows from \eqref{eq-second-cap-of-lagrangian-l-p} that for all $p>9/2$
\[ \begin{aligned}
{\rm Vol}({\mathbb X}_p) &\leq {\rm Vol} ({\mathbb X}_{\infty}) = \pi^2 \leq \tfrac{1}{2}\left(2\pi\left(\tfrac{1}{4}\right)^{2/9}\right)^2 = {\tfrac 1 2}c_2({\mathbb X}_{9/2})^2 \\ & \leq  {\tfrac 1 2}c_2({\mathbb X}_{p})^2 = \text{Vol}(B(c_2({\mathbb X}_p))).
\end{aligned}
\]
%

We now need to verify \eqref{eq:w2w3w4}. We claim that $\tau((\Omega_p)_1)$ and $\tau((\Omega_p)_{11})+\tau((\Omega_p)_{111})$ are increasing functions of $p$. For $p>9/2$, we now compute the $x$-intercept of the lines of slope $-1$, $-1/2$, $-1/3$ and $-1/4$, which intersect the curve \eqref{eq:curve}, but not $\R^2_{\ge 0}\setminus \Omega_p$. For $p\ge 25/2$, these lines are all tangent to \eqref{eq:curve}. Let $x_{-1}(p),x_{-1/2}(p),x_{-1/3}(p),x_{-1/4}(p)$ be the $x$-intercepts of these lines, see Figure \ref{fig:3}(a). We observe that
\[\begin{aligned}
\tau((\Omega_p)_1)&=x_{-1/2}(p)-x_{-1}(p),\\
\tau((\Omega_p)_{11})&=x_{-1/3}(p)-x_{-1/2}(p),\\
\tau((\Omega_p)_{111})&=x_{-1/4}(p)-x_{-1/3}(p).\end{aligned}\]
Now let $w_2(p)=\tau((\Omega_p)_1)$ and $d(p)=\tau((\Omega_p)_{11})+\tau((\Omega_p)_{111})$. So,
\begin{equation}\label{eq:w}
\begin{aligned}
w_2(p)&=x_{-1/2}(p)-x_{-1}(p),\\
d(p)&=x_{-1/4}(p)-x_{-1/2}(p).\end{aligned}\end{equation}
If we denote  by $v_{-1}(p)$, $v_{-1/2}(p)$ ,$v_{-1/3}(p)$ and $v_{-1/4}(p)$, the parameter values of $v$ where the lines defined above interesct the curve \eqref{eq:curve}, respectively, then
\begin{equation}\label{eq:v}
\begin{aligned}
v_{-1}(p)&=0,\\
v_{-1/2}(p)&=\left(g_p'\right)^{-1}(-2\pi/3),\\
v_{-1/4}(p)&=\left\{\begin{array}{ll}
\left(\frac{1}{4}\right)^{1/p}=\left(g_p'\right)^{-1}\left(-\sqrt{\frac{2}{p}}\pi\right),&\text{if } \frac{9}{2}<p<\frac{25}{2},\\
\left(g_p'\right)^{-1}(-2\pi/5),&\text{if }p\ge \frac{25}{2}.\\
\end{array}\right.
\end{aligned}
\end{equation}
Since $g_p$ is a concave function,
\begin{equation}\label{eq:vinc}v_{-1}(p)<v_{-1/2}(p)\le v_{-1/3}(p)\le v_{-1/4}(p).\end{equation}
Moreover 
\begin{equation}\label{eq:v2}\begin{aligned}
x_{-1}(p)&=2 g_p(0)=A(p),\\
x_{-1/2}(p)&=2\pi v_{-1/2}(p)+3g_p(v_{-1/2}(p)),\\
x_{-1/4}(p)&=2\pi v_{-1/4}(p)+5g_p(v_{-1/4}(p)).\\
\end{aligned}\end{equation}
In order to prove that the functions $w_2(p)$ and $d(p)$ are increasing, note first that they are continuous in $(9/2,\infty]$ and differentiable in $(9/2,25/2)\cup(25/2,\infty)$. So it suffices to show that
\begin{equation}\label{eq:inc}w_2'(p),d'(p)>0,\quad\text{for all }p\in(9/2,25/2)\cup(25/2,\infty).\end{equation}
It follows from \eqref{eq:dg2} that for $v<(1/4)^{1/p}$, the function $p\mapsto g_p'(v)$ is increasing. Since the function $(p,v)\mapsto g_p(v)$ is $C^\infty$ in the interior of its domain, we obtain
\[\frac{\partial}{\partial v}\frac{\partial}{\partial p}g_p(v)=\frac{\partial}{\partial p}g_p'(v)>0,\]
Hence for a fixed $p$, the function $v\mapsto\frac{\partial}{\partial p}g_p(v)$ is increasing. Further, differentiating \eqref{eq:g} for $v>0$ we obtain
\begin{equation}\label{eq:dgp}
\begin{aligned}
\frac{\partial}{\partial p} g_p(v)&=\int_{\left(\frac{1}{2}-\sqrt{\frac{1}{4}-v^p}\right)^{1/p}}^{\left(\frac{1}{2}+\sqrt{\frac{1}{4}-v^p}\right)^{1/p}}\frac{(1-r^p)^{2/p}}{\left((1-r^p)^{2/p}-\frac{v^2}{r^2}\right)^{1/2}}\left(\ln\frac{1}{(1-r^p)^{\frac{2}{p^2}}r^{\frac{2r^p}{p(1-r^p)}}}\right)\,dr &>0,
\end{aligned}
\end{equation}
since $r\in(0,1)$.
So
\[\begin{array}{rcccl}w_2'(p)&=&x_{-1/2}'(p)-x_{-1}'(p)&=&3\frac{\partial}{\partial p}g_p(v_{-1/2}(p))-2\frac{\partial}{\partial p}g_p(0)\\
&>&2\left(\frac{\partial}{\partial p}g_p(v_{-1/2}(p))-\frac{\partial}{\partial p}g_p(0)\right)&>&0.\end{array}\]
Let $p\in(9/2,15/2)$. It follows from \eqref{eq:v} that $v_{-1/4}'(p)>0$ and that  $g_p'(v_{-1/4}(p))>-2\pi/3$. Using \eqref{eq:w}, \eqref{eq:vinc}, \eqref{eq:v2} and \eqref{eq:dgp} we obtain
\[\begin{aligned} d'(p)&=x_{-1/4}'(p)-x_{-1/2}'(p)\\&=2\pi v_{-1/4}'(p)+5g_p'(v_{-1/4}(p))v_{-1/4}'(p)+ 5 \frac{\partial}{\partial p} g_p(v_{-1/4}(p))-3\frac{\partial}{\partial p}g_p(v_{-1/2}(p))\\
&>5 \frac{\partial}{\partial p} g_p(v_{-1/4}(p))-3\frac{\partial}{\partial p}g_p(v_{-1/2}(p))\\&> 3\left(\frac{\partial}{\partial p} g_p(v_{-1/4}(p))-\frac{\partial}{\partial p}g_p(v_{-1/2}(p))\right)\ge 0.\end{aligned}\]
Let $p\in(15/2/\infty)$. Then, by \eqref{eq:w}, \eqref{eq:v}, \eqref{eq:vinc}, \eqref{eq:v2} and \eqref{eq:dgp} 
\[\begin{aligned} d'(p)&=x_{-1/4}'(p)-x_{-1/2}'(p)\\&=2\pi v_{-1/4}'(p)+5g_p'(v_{-1/4}(p))v_{-1/4}'(p)+ 5 \frac{\partial}{\partial p} g_p(v_{-1/4}(p))-3\frac{\partial}{\partial p}g_p(v_{-1/2}(p))\\
&=5 \frac{\partial}{\partial p} g_p(v_{-1/4}(p))-3\frac{\partial}{\partial p}g_p(v_{-1/2}(p))\\&> 3\left(\frac{\partial}{\partial p} g_p(v_{-1/4}(p))-\frac{\partial}{\partial p}g_p(v_{-1/2}(p))\right)\ge 0.\end{aligned}\]
We conclude that $w_2(p)$ and $d(p)$ are increasing. A simple calculation using \eqref{eq:curve2} shows that
\begin{equation*}\label{eq:dinf}
d(\infty)=10\sin \frac{\pi}{5}-6\sin\frac{\pi}{3}=5\sqrt{\frac{5-\sqrt{5}}{2}}-3\sqrt{3}<0.69.
\end{equation*}
Moreover it follows from \eqref{eq:area-p-norm} that
\begin{equation*}\label{eq:c292}
w_2(9/2)=2\pi\left(\frac{1}{4}\right)^{2/9}-A\left(\frac{9}{2}\right)=2\pi\left(\frac{1}{4}\right)^{2/9}-\frac{4\cdot\Gamma\left(\frac{11}{9}\right)^2}{\Gamma\left(\frac{13}{9}\right)}>0.85.
\end{equation*}
So for any $p\in[9/2,\infty]$,
\[d(p)\le d(\infty)<w_2(9/2)\le w_2(p).\]
Hence, \eqref{eq:w2w3w4} is satisfied, and Proposition \ref{prop:flex} implies that \[X_{\Omega_p}\hookrightarrow B(c_2(X_{\Omega_p})).\]
\end{proof}

\begin{figure}
\begin{subfigure}{0.5\textwidth}
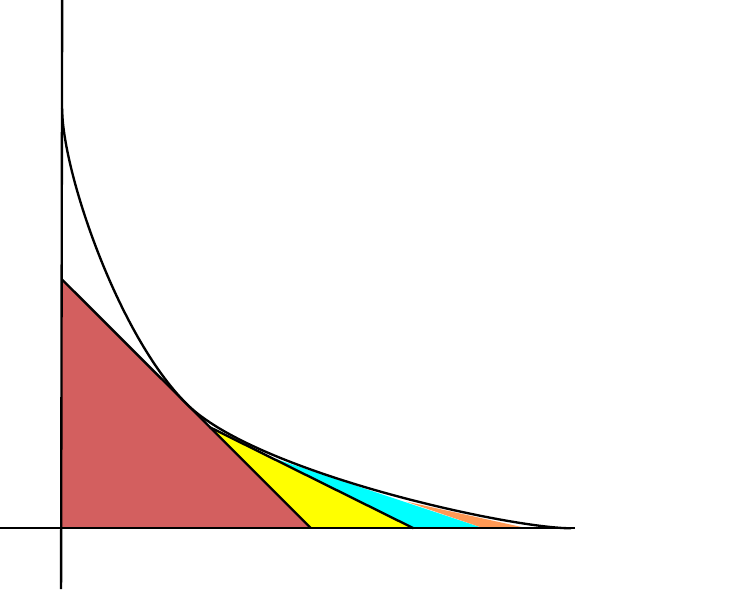
\caption{}
\end{subfigure}
\begin{subfigure}{0.5\textwidth}
\begingroup%
  \makeatletter%
  \providecommand\color[2][]{%
    \errmessage{(Inkscape) Color is used for the text in Inkscape, but the package 'color.sty' is not loaded}%
    \renewcommand\color[2][]{}%
  }%
  \providecommand\transparent[1]{%
    \errmessage{(Inkscape) Transparency is used (non-zero) for the text in Inkscape, but the package 'transparent.sty' is not loaded}%
    \renewcommand\transparent[1]{}%
  }%
  \providecommand\rotatebox[2]{#2}%
  \newcommand*\fsize{\dimexpr\f@size pt\relax}%
  \newcommand*\lineheight[1]{\fontsize{\fsize}{#1\fsize}\selectfont}%
  \ifx\svgwidth\undefined%
    \setlength{\unitlength}{269.99159984bp}%
    \ifx\svgscale\undefined%
      \relax%
    \else%
      \setlength{\unitlength}{\unitlength * \real{\svgscale}}%
    \fi%
  \else%
    \setlength{\unitlength}{\svgwidth}%
  \fi%
  \global\let\svgwidth\undefined%
  \global\let\svgscale\undefined%
  \makeatother%
  \begin{picture}(1,0.58335152)%
    \lineheight{1}%
    \setlength\tabcolsep{0pt}%
    \put(0,0){\includegraphics[width=\unitlength,page=1]{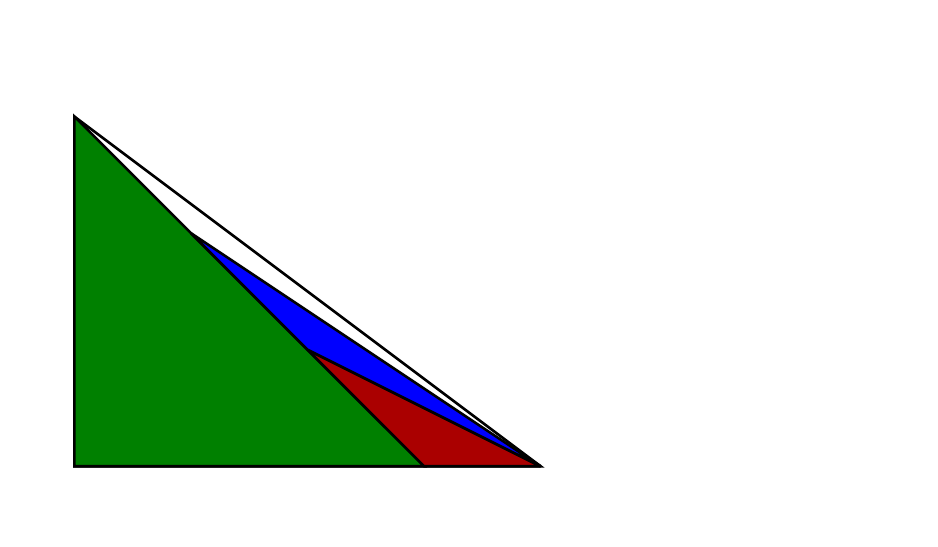}}%
    \put(0.01719633,0.49009467){\color[rgb]{0,0,0}\makebox(0,0)[lt]{\lineheight{1.25}\smash{\begin{tabular}[t]{l}$\frac{1}{2}$\end{tabular}}}}%
    \put(0.58935563,0.02637942){\color[rgb]{0,0,0}\makebox(0,0)[lt]{\lineheight{1.25}\smash{\begin{tabular}[t]{l}$\frac{2}{3}$\end{tabular}}}}%
  \end{picture}%
\endgroup%

\caption{}
\end{subfigure}
\caption{}
\label{fig:3}
\end{figure}

\subsection{The symplectic $\ell^p$-sum of two discs}
Here we prove Theorem \ref{thm:symplp} and Proposition \ref{prop:lpellip}.

\begin{proof}[{\bf Proof of Theorem \ref{thm:symplp}}]
It is clear from the definition that ${\mathbb B}_p({\mathbb C}^2)$ is a concave toric domain for $0<p<2$, and a convex toric domain for $p>2$. Thus, it follows from Lemma \ref{lemma:c1c2} that
\begin{align} 
c_1({\mathbb B}_p({\mathbb C}^2))&=\min(1,2^{1-2/p}), \ {\rm if \ } p>0, \label{eq-c1-ell-p-sum} \\
c_2({\mathbb B}_p({\mathbb C}^2))&=\left\{\begin{array}{ll}2^{1-2/p},& \text{ if }p\ge 2\\
\left({1 + 2^{\frac{p}{p - 2}}}\right)^{1-2/p}, &\text{ if }1\le p<2. \label{eq-c2-ell-p-sum}  \end{array}\right.
\end{align}
A straight forward calculation shows that \begin{equation} \label{eq-inradius-symp-p-sum-1} B(c_1({\mathbb B}_p({\mathbb C}^2)))\subset {\mathbb B}_p({\mathbb C}^2),\text{ for all }p > 0,\end{equation} and 
\begin{equation} \label{eq-outradius-symp-p-sum-1}
{\mathbb B}_p({\mathbb C}^2) \subset B(c_2({\mathbb B}_p({\mathbb C}^2))),\text{ for }p\ge 2.\end{equation}
The combination of~\eqref{eq-c1-ell-p-sum},~\eqref{eq-c2-ell-p-sum},~\eqref{eq-inradius-symp-p-sum-1},~\eqref{eq-outradius-symp-p-sum-1} implies parts (a) and (b) of the theorem.
In order to prove (c), we apply Proposition \ref{prop:flex}. Let $p\in[1,2)$ and let $\tilde{\Omega}_p=\mu({\mathbb B}_p({\mathbb C}^2))$, where $\mu$ is the moment map. 
Note that $\tilde{\Omega}_p \subset {\mathbb R}^2_{\geq 0}$ is the region bounded by the coordinate axes and the curve
$$ x^{p/2}+y^{p/2}=1 \ \ ({\rm and \ } \max \{ \|x\|, \|y\| \} =1 \ {\rm for \ } p=\infty).$$
As in the proof of Proposition \ref{prop:emb} above, we can define $x_{-1}(p)$, $x_{-1/2}(p)$, $x_{-1/3}(p)$ and $x_{-1/4}(p)$ for $\tilde{\Omega}_p$, and  a simple computation shows that for $n \geq 1$ 
\begin{equation}\label{eq:xbp}
x_{-1/n}(p)=\left(\frac{n^\frac{p}{2-p}}{1+n^\frac{p}{2-p}}\right)^{\frac{2-p}{p}}.\end{equation}
Thus, as in the proof of Proposition \ref{prop:emb}, one has \begin{equation}\label{eq:cbp}\begin{aligned}
c_2(p)&=x_{-1/2}(p),\\
w_2(p)&=\tau((\tilde{\Omega}_p)_1)=x_{-1/2}(p)-x_{-1}(p),\\
d(p)&=\tau((\tilde{\Omega}_p)_{11})+\tau((\tilde{\Omega}_p)_{111})=x_{-1/4}(p)-x_{-1/2}(p).\end{aligned}
\end{equation}
Moreover, a direct computation gives
\begin{equation}\label{eq:volbp}
\text{vol}(B_p(1,1))=\int_0^1 (1-t^{p/2})^{2/p}\,dt={\frac 1 p}B \left ({\frac 2 p},{\frac 2 p} \right ),\end{equation}
where here $B(\alpha,\beta)$ stands for the Euler beta function.
It follows from \eqref{eq:xbp}, \eqref{eq:cbp} and \eqref{eq:volbp} that the assumptions for Lemma \ref{prop:flex} are:
\begin{align}
\frac{1}{p}B\left(\frac{2}{p},\frac{2}{p}\right)&\le\frac{1}{2}\left(\left(\frac{2^\frac{p}{2-p}}{1+2^\frac{p}{2-p}}\right)^{\frac{2}{p}-1}\right)^2\label{eq:1}\\
\left(\frac{4^\frac{p}{2-p}}{1+4^\frac{p}{2-p}}\right)^{\frac{2-p}{p}}&-\left(\frac{2^{\frac{p}{2-p}}}{1 + 2^{\frac{p}{2-p}}}\right)^{\frac{2-p}{p}} \le\left(\frac{2^{\frac{p}{2-p}}}{1 + 2^{\frac{p}{2-p}}}\right)^{\frac{2-p}{p}}-\left(\frac{1}{2}\right)^{\frac{2-p}{p}}\label{eq:2}.
\end{align}
We will first prove \eqref{eq:1}. We claim that the function $x\mapsto x B(x,x)$ is convex in $(1,2)$. In fact, for $x\in(1,2)$,
\begin{equation}\label{eq:int}\begin{aligned}\frac{d^2}{dx^2} (x B(x,x))&=\int_0^1\frac{d^2}{dx^2}\left(x (t(1-t))^{x-1}\right)\,dt\\&=\int_0^1\left(t(1-t)\right)^{x-1}\left(\ln(t(1-t))+(\ln(t(1-t)))^2x\right)\,dt.\end{aligned}\end{equation}
The maximum of the function $t\mapsto -1/\ln(t(1-t))$ is $1/\ln(4)$, which is smaller than 1, so the integrand in \eqref{eq:int} is positive. Therefore $x\mapsto x B(x,x)$ is convex in $(1,2)$. Hence, for all $x\in(1,2)$,
\begin{equation}\label{eq:bconv}
xB(x,x)\le B(1,1)+\left(2B(2,2)-B(1,1)\right)(x-1)=\frac{5}{3}-\frac{2x}{3}.
\end{equation}
Now it is a simple calculus problem to check that for $x\in(1,2)$,
\begin{equation}\label{eq:ineqline}
\frac{5}{3}-\frac{2x}{3}\le \frac{4}{(1+2^{\frac{1}{x-1}})^{2(x-1)}}=\left(\frac{2^\frac{1}{x-1}}{1+2^{\frac{1}{x-1}}}\right)^{2(x-1)}.
\end{equation}
Combining \eqref{eq:bconv} and \eqref{eq:ineqline} for $x=2/p$, we obtain \eqref{eq:1}.

\medskip

We now prove \eqref{eq:2} for $1\le p<2$. It is clear that \[2^{\frac{3p}{2-p}}-3\cdot 2^{\frac{2p}{2-p}}+3\cdot 2^{\frac{p}{2-p}}-1=\left(2^{\frac{p}{2-p}}-1\right)^3>0.\] Consequently \[ \frac{2^{\frac{p}{2-p}}}{1+2^{\frac{p}{2-p}}}>\frac{1+3\cdot 2^{\frac{2p}{2-p}}}{4\left(1+2^{\frac{2p}{2-p}}\right)}=\frac{1}{2}\left(\frac{1}{2}+\frac{4^{\frac{p}{2-p}}}{1+4^{\frac{p}{2-p}}}\right).\]
Since the function $x\mapsto x^{\frac{2-p}{p}}$ is concave, it follows that
\[\left(\frac{2^{\frac{p}{2-p}}}{1+2^{\frac{p}{2-p}}}\right)^{\frac{2-p}{p}}>\left( \frac{1}{2}\left(\frac{1}{2}+\frac{4^{\frac{p}{2-p}}}{1+4^{\frac{p}{2-p}}}\right)\right)^{\frac{2-p}{p}}\ge\frac{1}{2}\left(\left(\frac{1}{2}\right)^{\frac{2-p}{p}}+\left(\frac{4^{\frac{p}{2-p}}}{1+4^{\frac{p}{2-p}}}\right)^{\frac{2-p}{p}}\right)\]
and so \[\left(\frac{1}{2}\right)^{\frac{2-p}{p}}+\left(\frac{4^{\frac{p}{2-p}}}{1+4^{\frac{p}{2-p}}}\right)^{\frac{2-p}{p}}\le 2\left(\frac{2^{\frac{p}{2-p}}}{1+2^{\frac{p}{2-p}}}\right)^{\frac{2-p}{p}}.\]
Therefore \eqref{eq:2} holds.
\end{proof}

\begin{proof}[{\bf Proof of Proposition \ref{prop:lpellip}}]
It is enough to prove that ${\mathbb B}_1({\mathbb C}^2)\hookrightarrow E(1/2,2/3)$. The domain ${\mathbb B}_1({\mathbb C}^2)$ is a symmetric concave toric domain, and a direct computation shows that the first few numbers of its weight sequence are
\[\begin{aligned}&w_1=\frac{1}{2},&\;&w_2=\frac{1}{6},&\;&w_3=\frac{1}{6},&\;&w_4=\frac{1}{12},&\;&w_5=\frac{1}{12},&\;&w_6=\frac{1}{20},&\\&w_7=\frac{1}{20},&\;&w_8=\frac{1}{30},&\;&w_9=\frac{1}{30},&\;&w_{10}=\frac{1}{30},&\;&w_{11}=\frac{1}{30}.&&&\end{aligned}\]
%
%
Note that one can easily fit the ball $B(1/2)$ and two balls $B(1/6)$ into $E(1/2,2/3)$, see Figure~\ref{fig:3}(b). The remaining domain is equivalent to the ball $B(1/6)$ under an $SL(2,\Z)$ transformation. So for $N\ge 4$,
\begin{equation}\label{eq:embeq}\bigsqcup_{i=1}^N B(w_i)\hookrightarrow E(1/2,2/3)\iff \bigsqcup_{i=4}^N B(w_i)\hookrightarrow B(1/6).\end{equation}
We now consider the ordered vector
\[(1/6;1/12,1/12,1/20,1/20,1/30,1/30,1/30,1/30,\dots,w_{N}).\]
After appling one Cremona move (and re-ordering), we obtain the vector
\[(7/60;1/20,1/30,1/30,1/30,1/30,1/30,1/30,1/30,\dots,w_N,0),\]
which is reduced. Thus, Theorem \ref{thm:cremona} implies that the embeddings in \eqref{eq:embeq} exist, and therefore from Theorem \ref{thm:cgcc} we conclude that  ${\mathbb B}_1({\mathbb C}^2) \hookrightarrow E(1/2,2/3)$, as required.
\end{proof}

\bibliographystyle{siam}
\bibliography{bib}

\end{document}